\documentclass[final,leqno,onefignum,onetabnum]{siamltex1213}
\usepackage{amsfonts}
\usepackage{amsmath,booktabs,ctable,threeparttable}
\usepackage{amssymb,amsfonts,boxedminipage}
\usepackage{algorithm}
\usepackage{algorithmic}
\usepackage{chngcntr}
\usepackage{subfigure}
\usepackage{qbordermatrix}
\usepackage[utf8]{inputenc}
\usepackage{cleveref}
\usepackage{lineno}

\crefname{section}{§}{§§}
\Crefname{section}{§}{§§}
\counterwithout{algorithm}{section}
\newtheorem{Def}{Definition}[section]

\newtheorem{Lem}{Lemma}[section]

\title{The joint bidiagonalization method for large GSVD computations
in finite precision\thanks{This
		work was supported in part by
		the National Science Foundation of China (No. 12171273)}}

\author{Zhongxiao Jia\thanks{Department of Mathematical Sciences, Tsinghua
		University, 100084 Beijing, China. (\email{jiazx@tsinghua.edu.cn})} \and
	Haibo Li\thanks{Department of Mathematical Sciences, Tsinghua
	University, 100084 Beijing,  China.
	 (\email{lee12358@163.com})}}

\begin{document}
	\maketitle
	
\begin{abstract}
The joint bidiagonalization (JBD) method has been used to compute
some extreme generalized singular values and vectors of a large
regular matrix pair $\{A,L\}$, where we propose
three approaches to compute approximate generalized singular values
and vectors. We make a numerical analysis of the
underlying JBD process and establish relationships between it and
two mathematically equivalent Lanczos bidiagonalizations in finite precision.
Based on the results of numerical
analysis, we investigate the convergence of the approximate
generalized singular values and vectors of $\{A,L\}$. The results show that,
under some mild conditions,
the semiorthogonality of Lanczos type vectors suffices to
deliver approximate generalized singular values with the same accuracy
as the full orthogonality does,
meaning that it is only necessary to seek for efficient semiorthogonalization
strategies for the JBD process. We also establish a sharp bound
for the residual norm of an approximate generalized singular value and corresponding
approximate right generalized singular vectors, which can reliably estimate the residual norm
without explicitly computing the approximate right generalized
singular vectors before the convergence occurs.

\end{abstract}



\section{Introduction}\label{sec1}

In \cite{Zha1996}, Zha presents a joint bidiagonalization (JBD) process
that jointly bidiagonalizes a large sparse or
structured matrix pair $\{A,L\}$ to upper diagonal forms
successively, where $A\in\mathbb{R}^{m\times n}$ and $L\in\mathbb{R}^{p\times n}$.
He exploits the JBD process to compute
a few extreme generalized singular values and vectors of
$\{A,L\}$ \cite{Paige1981, Van1976, Van1985}.
Kilmer \textit{et al.} \cite{Kilmer2007} develop a variant of the JBD
process that jointly reduces $\{A,L\}$ to lower and upper
bidiagonal forms simultaneously. Besides
the computation of a few extreme generalized singular
value decomposition (GSVD) components, this variant is
used to solve large scale linear discrete ill-posed problems with
general-form regularization, where $L$ is the regularization matrix
\cite{Hansen1998,Hansen2010,JiaYang2020,Kilmer2007}.

The JBD process of $\{A,L\}$ is mathematically equivalent to Lanczos
bidiagonalizations \cite{Bjorck1996,Golub2013,Paige1982} of two certain
related matrices, as we will describe next.
Consider the compact QR factorization of the stacked matrix:
\begin{equation}\label{1.1}
\begin{pmatrix}
A \\
L
\end{pmatrix} = QR =
\begin{pmatrix}
Q_{A} \\
Q_{L}
\end{pmatrix}R ,
\end{equation}
where $Q \in \mathbb{R}^{(m+p)\times n}$ is column orthonormal and
$R\in \mathbb{R}^{n\times n}$ is upper triangular, and
$Q_{A}\in\mathbb{R}^{m\times n}$ and $Q_{L}\in\mathbb{R}^{p\times n}$.
If ${\rm rank}((A^T,L^T)^T)=n$ with ${\rm rank}(\cdot)$
the rank of a matrix, we call the pair $\{A,L\}$ is regular;
in this case, $R$ is nonsingular.

Applying the BIDIAG-1 algorithm and the
BIDIAG-2 algorithm in \cite{Paige1982}, which are the lower
and upper Lanczos bidiagonalizations, to $Q_{A}$ and $Q_{L}$,
respectively, we can reduce $Q_{A}$ and $Q_{L}$ to the following
lower and upper bidiagonal matrices:
\begin{equation}\label{1.2}
B_{k}=\begin{pmatrix}
\alpha_{1} & & & \\
\beta_{2} &\alpha_{2} & & \\
&\beta_{3} &\ddots & \\
& &\ddots &\alpha_{k} \\
& & &\beta_{k+1}
\end{pmatrix}\in  \mathbb{R}^{(k+1)\times k},
\widehat{B}_{k}=\begin{pmatrix}
\hat{\alpha}_{1} &\hat{\beta}_{1} & & \\
&\hat{\alpha}_{2} &\ddots & \\
& &\ddots &\hat{\beta}_{k-1} \\
& & &\hat{\alpha}_{k}
\end{pmatrix}\in  \mathbb{R}^{k\times k} .
\end{equation}
The above two algorithms produce the four orthonormal matrices
\begin{equation}\label{1.3}
U_{k+1}=(u_{1},\ldots,u_{k+1}) \in \mathbb{R}^{m\times (k+1)}, \ \
V_{k}=(v_{1},\ldots,v_{k}) \in \mathbb{R}^{n\times k}
\end{equation}
and
\begin{equation}\label{1.4}
 \widehat{U}_{k}=(\hat{u}_{1},\ldots,\hat{u}_{k}) \in \mathbb{R}^{p\times k}, \ \
\widehat{V}_{k}=(\hat{v}_{1},\ldots,\hat{v}_{k}) \in \mathbb{R}^{n\times k}.
\end{equation}

The BIDIAG-1 and BIDIAG-2 algorithms can be related by taking the starting vector
$\hat{v}_{1}= v_{1}$ in BIDIAG-2.
It has been proved in \cite{Kilmer2007,Zha1996} that the
Lanczos vectors $\hat{v}_{i}$ and the elements $\hat{\alpha}_i,\hat{\beta}_{i}$ of
$\widehat{B}_{k}$ have the following relations:
\begin{equation}\label{1.5}
\hat{v}_{i+1} = (-1)^{i}v_{i+1} , \ \
\hat{\beta}_{i} = \alpha_{i+1}\beta_{i+1}/\hat{\alpha}_{i} .
\end{equation}
For $A$ and $L$ large, the explicit QR
factorization \eqref{1.1} is generally impractical
due to the excessive storage and/or computational cost. It
can be avoided by solving a least squares problem with $(A^{T}, L^{T})^{T}$
as the coefficient matrix at each iteration $i$, $i=1,2,\ldots,k$.
Zha \cite{Zha1996} and Kilmer, Hansen and Espanol
\cite{Kilmer2007} propose the JBD process that successively reduces $\{A,L\}$ to
$\{B_{k}, \bar{B}_{k}\}$, where $\bar{B}_k=\widehat{B}_kD_k$
with $D_k=\diag (1,-1,\dots ,(-1)^{k-1})\in\mathbb{R}^{k\times k}$, which
will be described in the next section.
The $k$-step JBD process explicitly computes three orthonormal matrices
$U_{k+1}\in \mathbb{R}^{m\times (k+1)}$,
$\widehat{U}_{k}\in \mathbb{R}^{p\times k}$, and
$\widetilde{V}_{k}\in \mathbb{R}^{(m+n)\times k}$, and the lower and upper
bidiagonal matrices $B_{k}$ and $\bar{B}_{k}$. The matrices
$V_{k}$ and $\widehat{V}_k$ in the BIDIAG-1 and BIDIAG-2 algorithms
are related to $\widetilde{V}_{k}$
by $V_{k}=Q^{T}\widetilde{V}_{k}$ and $\widehat{V}_{k}=V_kD_k$.
Therefore, the JBD process on $\{A, L\}$ is mathematically equivalent
to the joint lower and upper Lanczos bidiagonalizations of $Q_A$
and $Q_L$ when taking $\hat{v}_1=v_1$.

The lower bidiagonal $B_{k}$ is the Ritz--Galerkin projection of $Q_{A}$ on the
left subspace ${\rm span}(U_{k+1})$ and the right subspace ${\rm span}(V_{k})$, while
the upper bidiagonal $\bar{B}_{k}$ is the Ritz--Galerkin projection of
$Q_{L}$ on the left and right subspaces ${\rm span}(\widehat{U}_{k})$ and
${\rm span}(V_{k})$, where ${\rm span}(\cdot)$ denotes
the subspace spanned by the columns of a matrix. Therefore, the extreme
singular values of $Q_A$ or $Q_L$  can be approximated by
those of $B_{k}$ or $\widehat{B}_{k}$, and the extreme
generalized singular values of $\{A,L\}$ can be approximated by
those of $\{B_k,\bar{B}_k\}$ since the generalized singular values of
$\{A,L\}$ are identical to those of $\{Q_A,Q_L\}$ \cite{JiaYang2020,Zha1996}.

Due to the influence of rounding errors, we have
numerically observed that the orthogonality of the three
sets of basis vectors, also called Lanczos vectors conventionally,
computed by the JBD process loses gradually. This is a typical phenomenon
in Lanczos type algorithms, such as the
symmetric Lanczos process \cite{Lanczos1950} and the Lanczos
bidiagonalization process \cite{Larsen1998}. The loss of
orthogonality of Lanczos vectors leads to a delay of
convergence of some extreme eigenvalues
and the appearance of spurious computed Ritz values, i.e., ghost Ritz values,
\cite{Meurant2006,Paige1971,Paige1972,Paige1980}. To fix this deficiency,
several reorthogonalization strategies have been proposed to
maintain some level of orthogonality of the computed
Lanczos vectors in order to ensure the
convergence of the computed Ritz values \cite{Parlett1979,Simon1984a,Simon1984b}.
Particularly, Simon \cite{Simon1984a} proves that the semiorthogonality of
Lanczos vectors suffices to guarantee that the computed
Ritz values have the same accuracy as the full orthogonality does and
avoid spurious computed Ritz values.
The above results on the symmetric Lanczos process have been
adapted by Larsen \cite{Larsen1998} to Lanczos bidiagonalization,
based on which he proposes an efficient partial reorthogonalization strategy.
Later on, Simon and Zha in \cite{Simon2000} propose a one-sided
reorthogonalization strategy on the computed right Lanczos vectors.
Barlow \cite{Barlow2013} makes
a backward error analysis of the one-sided reorthogonalization scheme
and proves that Lanczos bidiagonalization
applied to a matrix $C$ in finite precision produces Krylov
subspaces generated by a nearby matrix $C+E$, where $E$ is an error matrix
depending on the orthogonality level of the computed
right Lanczos vectors.

Denote the unit roundoff by $\epsilon$.
In the presence of rounding errors, among many others, a central concern
is whether or not the JBD process for computing $U_{k+1}$, $V_{k}$
and $B_{k}$ is equivalent to the standard
lower Lanczos bidiagonalization of $Q_{A}$
with the rounding error $O(\epsilon)$ and whether or not the process for
computing $\widehat{U}_{k}$, $\widehat{V}_{k}$ and $\widehat{B}_{k}$ is
equivalent to the upper Lanczos bidiagonalization of $Q_{L}$
with the rounding error $O(\epsilon)$. There has been yet no
result on the finite precision behavior of the JBD process. In this paper,
we will focus on it and, based on some underlying round-off error
models and results,
make a numerical analysis of the JBD process. We will
derive a number of properties of the JBD process
in finite precision.
Our contributions mainly consist of the following three parts.

First, we will show that the equivalence of the JBD process and
standard lower and upper Lanczos bidiagonalizations
does not hold in finite precision unconditionally.
That is, the finite precision forms resulting from the JBD process
may be no longer the corresponding ones of standard
Lanczos bidiagonalizations if there are no additional conditions.
We will investigate what a role rounding errors play
in the loss of this equivalence and in what way rounding errors are amplified.

Second, in finite precision, we will show that the orthogonality
levels of $U_{k+1}$, $\widetilde{V}_{k}$ and $\widehat{U}_{k}$
are closely related and those of $\widehat{V}_k$ and $V_k$ interact too.
In particular, we derive an
upper bound for the orthogonality level of $\widehat{U}_{k}$,
which is shown to be controlled by not only the orthogonality levels of
$\widetilde{V}_{k}$ and ${U}_{k+1}$ but also a gradually
growing quantity $\|\widehat{B}_{k}^{-1}\|$. The result indicates
that the orthogonality level of $\widehat{U}_{k}$ is similar to
those of $U_{k+1}$ and $\widetilde{V}_{k}$, provided that
$\widehat{B}_{k}$ is not ill conditioned.
Therefore, when designing a reorthogonalization
strategy for the JBD process, one reorthogonalizes only $u_{i}$
and $\tilde{v}_{i}$ but $\hat{u}_{i}$, which can save considerable
reorthogonalization work.

Third, we shall show how to make use of the JBD process to
compute extreme GSVD components of $\{A,L\}$, leading to
the JBD method for the large GSVD computation. Precisely,
we will propose three approaches to compute
approximate generalized singular values and approximate right
generalized singular vectors by exploiting the SVDs of $B_k$
and $\bar{B}_k$ and the GSVD of $\{B_k,\bar{B}_k\}$, respectively;
we will present two approaches to compute approximate left
generalized singular vectors of $A$ and $L$ via either
the left singular vectors of $B_k$ and $\bar{B}_k$
simultaneously or the left generalized singular vectors of
$\{B_k,\bar{B}_k\}$. We will investigate the convergence
of the approximate generalized singular values that are computed by
the singular values of $B_{k}$ or $\bar{B}_{k}$. Similarly to
Lanczos bidiagonalization, the loss of orthogonality of
Lanczos type vectors computed by the JBD process leads to a delay of the
convergence of the computed Ritz values and the appearance of
spurious copies. We show that, under the assumptions
that $\|B_k^{-1}\|$ and $\|\widehat{B}_k^{-1}\|$ are
modest uniformly with $k$, the semiorthogonality of Lanczos type
vectors suffices to avoid spurious copies and guarantees that
the approximate generalized singular values have the same accuracy
as the full orthogonality does. Here the semiorthogonality
means that the absolute value of inner product of
two unit length vectors is at the level of $O(\epsilon^{1/2})$,
in contrast to the full orthogonality level $O(\epsilon)$. Therefore,
the semiorthogonality of Lanczos vectors suffices for the JBD method.
In the meantime,
we study the residual norm of an approximate generalized singular value and
approximate right generalized singular vector,
whose size is used to design a stopping tolerance for the JBD method.
In finite precision, we derive an upper bound
for the residual norm, and show that this upper bound can replace
the residual norm to design a reliable stopping criterion
without explicitly computing approximate right generalized singular vectors
before the convergence occurs. We only compute
the approximate right generalized singular vectors
by solving certain consistent least squares problems with the coefficient matrix
$(A^T,L^T)^T$ at the convergence rather than doing so at each iteration.

The paper is organized as follows. In Section \ref{sec2}, we review
the GSVD of $\{A,L\}$ and describe the JBD process
in exact arithmetic. In Section \ref{sec3},
we make a numerical analysis of the JBD process in finite
precision. We establish relationships between the JBD
process and two lower and upper Lanczos bidiagonalizations, and
investigate interactions of orthogonality levels of the computed
Lanczos type vectors. In Section \ref{sec4}, we describe the JBD
method for computing a number of extreme generalized singular values and
vectors of $\{A, L\}$, and discuss the convergence
and stopping criteria. In Section \ref{sec5},
we report numerical experiments to confirm our results. Finally, we
conclude the paper with some remarks and future work in Section \ref{sec6}.

Throughout the paper, we denote by $I_k$ the identity matrix of order $k$,
by $0_k$ and $0_{k\times l}$ the $k$-dimensional zero vector and
the $k\times l$ zero matrix, respectively. The transpose of a matrix
$C$ is denoted by $C^{T}$, and $\|\cdot \|$ is the 2-norm of a matrix.

\section{GSVD and the JBD process}\label{sec2}

We describe the GSVD of $\{A,L\}$ and the JBD process with some of its
basic properties. Let the compact QR
factorization of $(A^{T}, L^{T})^{T}$ be defined as \eqref{1.1} and
\begin{equation}\label{2.1}
Q_{A} = P_{A}C_{A}W^{T} , \ \  Q_{L} = P_{L}S_{L}W^{T}
\end{equation}
be the CS decomposition of the matrix pair
$\{Q_{A}, Q_{L} \}$ \cite[\S 2.5.4]{Golub2013}, where $P_{A}\in \mathbb{R}^{m\times m}$,
$P_{L}\in \mathbb{R}^{p\times p}$ and $W\in\mathbb{R}^{n\times n}$
are orthogonal matrices, and $C_{A}\in\mathbb{R}^{m\times n}$ and
$S_{L}\in\mathbb{R}^{p\times n}$ are diagonal matrices (not necessarily square) satisfying
$C_{A}^{T}C_{A}+S_{L}^{T}S_{L}=I_{n}$.

Suppose that $\rank((A^{T}, L^{T})^{T})=r$.
It is shown in \cite{Paige1981} that $C_{A}$ and $S_{L}$ can be written as
$$C_{A} =
\bordermatrix*[()]{%
	\Sigma_{A}, & 0 & m \cr
	r &   n-r   \cr
} \ , \ \ \ \
S_{L} = \bordermatrix*[()]{%
	\Sigma_{L}, & 0 & p \cr
	r &   n-r   \cr
} ,$$
where
$$\Sigma_{A} =
\bordermatrix*[()]{%
	I_{q}  &  &  & q \cr
	&  C_{l}  &  & l \cr
	&  & O  & m-q-l \cr
	q & l & r-q-l
} , \ \ \ \
\Sigma_{L} =
\bordermatrix*[()]{%
	O  &  &  & p-r+q \cr
	&  S_{l}  &  & l \cr
	&  & I_{t}  & r-q-l \cr
	q & l & r-q-l
} .$$
Write $C_{l}=\diag(c_{q+1}, \dots, c_{q+l})$, \ $c_{q+1}\geq \cdots
\geq c_{q+l}>0$ and $S_{l}=\diag(s_{q+1}, \dots, s_{q+l})$, \ $0<s_{q+1}\leq
\cdots \leq s_{q+l}$. Then $c_{i}^{2}+s_{i}^{2}=1, \ i=q+1, \dots, q+l$,
and the generalized singular values of $\{A, L\}$ are
$$\underbrace{\infty, \dots, \infty}_{q}, \ \
\underbrace{c_{q+1}/s_{q+1}, \dots, c_{q+l}/s_{q+l}}_{l}, \ \
\underbrace{0, \dots, 0}_{t},
$$
where $t=r-q-l$.

To ease the presentation, in the sequel, we always assume that
$\{A,L\}$ is regular. Then $R$ in \eqref{1.1} is nonsingular and
the GSVD of $\{A, L\}$ is
\begin{equation}\label{2.2}
A = P_{A}C_{A}X^{-1} , \ \  L = P_{L}S_{L}X^{-1}
\end{equation}
with $X=R^{-1}W\in\mathbb{R}^{n\times n}$.
Let $X=(x_1,\ldots,x_n)$, $P_A=(p_1^A,\ldots,p_m^A)$ and
$P_L=(p_1^L,\ldots,p_p^L)$. We can write the GSVD
\eqref{2.2} in the vector form:
\begin{equation}\label{gsvdv}
  \left\{
  \begin{aligned}
  Ax_i&=c_i p_i^A,\\
  Lx_i&=s_i p_i^L,\\
  s_i A^Tp_i^A&=c_iL^Tp_i^L,
  \end{aligned}
  \right.
  \qquad
  i=1,\dots,n,
\end{equation}
where the $i$-th large generalized singular value
of $\{A,L\}$ is $c_{i}/s_{i}$, and the $i$-th corresponding
generalized singular vectors are $x_{i}$, $p^{A}_{i}$ and $p_{i}^{L}$,
respectively.
We call $x_{i}$ the right generalized singular vector,
$p^{A}_{i}$ and $p_{i}^{L}$ the left generalized singular
vectors corresponding to $c_{i}/s_{i}$.
Since $c_{i}/s_{i}=\infty$ and $c_{i}/s_{i}=0$ when $s_{i}=0$
and $s_{i}=1$, it is more convenient to use
pair $\{c_{i}, s_{i}\}$ to denote $c_{i}/s_{i}$. We also note that
each $x_i$ satisfies the normalization $x_i^T(A^TA+L^TL)x_i=1$.

We describe the JBD process \cite{Kilmer2007} as Algorithm \ref{alg1},
which, in exact arithmetic, corresponds to the lower and upper
Lanczos bidiagonalization of $Q_A$ and $Q_L$, respectively:
\begin{align}
&Q_AV_k=U_{k+1}B_k,\ \ Q_A^TU_{k+1}=V_kB_k^T+\alpha_{k+1}v_{k+1}e_{k+1}^T,\label{2.10}\\
&Q_L\widehat{V}_k=\widehat{U}_k\widehat{B}_k,\ \ Q_L^T\widehat{U}_k
=\widehat{V}_k\widehat{B}_k^T+\hat{\beta}_k\hat{v}_{k+1}e_k^T  \label{2.11}
\end{align}
with $\hat{v}_1=v_1$, where $e_{k}$ is the last column of $I_k$,
which are $k$-step lower and upper Lanczos bidiagonalization
processes of $Q_A$ and $Q_L$, respectively.

\begin{algorithm}[htb]
	\caption{The $k$-step JBD process}
	\begin{algorithmic}[1]\label{alg1}
		\STATE {Choose a nonzero starting vector $b \in \mathbb{R}^{m}$,
			and let $\beta_{1}u_{1}=b,\ \beta_{1}=\| b\|$ }
		\STATE {$\alpha_{1}\tilde{v}_{1}=QQ^{T}\begin{pmatrix}
			u_{1} \\
			0_{p}
			\end{pmatrix} $}
		\STATE { $\hat{\alpha}_{1}\hat{u}_{1}=\tilde{v}_{1}(m+1:m+p) $}
		\FOR{$i=1,2,\ldots,k,$}
		\STATE $\beta_{i+1}u_{i+1}=\tilde{v}_{i}(1:m)-\alpha_{i}u_{i} $
		\STATE $ \alpha_{i+1}\tilde{v}_{i+1}=
		QQ^{T}\begin{pmatrix}
		u_{i+1} \\
		0_{p}
		\end{pmatrix}-\beta_{i+1}\tilde{v}_{i} $
		\STATE $\hat{\beta}_{i}=(\alpha_{i+1}\beta_{i+1})/\hat{\alpha}_{i} $
		\STATE $\hat{\alpha}_{i+1}\hat{u}_{i+1}=
		(-1)^{i}\tilde{v}_{i+1}(m+1:m+p)-\hat{\beta}_{i}\hat{u}_{i} $
		\ENDFOR
	\end{algorithmic}
\end{algorithm}

At each iteration $i=1,2,\ldots,k$, Algorithm~\ref{alg1} needs to compute $QQ^{T} \begin{pmatrix}
u_i \\ 0_p
\end{pmatrix}$.
For $A$ and $L$ large, however, the compact QR
factorization \eqref{1.1} of $(A^{T}, L^{T})^{T}$ is generally impractical,
that is, both $Q$ and $R$ are not available. Let $\tilde{u}_i=\begin{pmatrix}
u_i \\ 0_p
\end{pmatrix}$. Since $QQ^T\tilde{u}_i$ is nothing but the orthogonal projection
of $\tilde{u}_i$ onto the column space of $(A^{T}, L^{T})^{T}$, we have $QQ^T\tilde{u}_i=\begin{pmatrix}
A \\ L
\end{pmatrix}\tilde{x}_i$, where
\begin{equation}\label{2.3}
\tilde{x}_i=\arg\min_{\tilde{x}\in \mathbb{R}^n}
\left\|\begin{pmatrix}
A \\ L
\end{pmatrix}
\tilde{x}-\tilde{u}_i\right\|.
\end{equation}
This large scale least squares problem can be solved by an
iterative solver, e.g., the most commonly used LSQR algorithm \cite{Paige1982}.

In exact arithmetic, the $k$-step JBD process produces the two
bidiagonal matrices $B_{k}$, $\widehat{B}_{k}$ and three
orthonormal matrices $U_{k+1}$, $\widehat{U}_{k}$ in \eqref{1.3}--\eqref{1.4} and
\begin{align}\label{2.4}
\widetilde{V}_{k}=(\tilde{v}_{1},\ldots,\tilde{v}_{k}) \in \mathbb{R}^{(m+p)\times k},
\end{align}
where $\tilde{v}_{i}=Qv_{i}$
with $v_i$ the $i$-th column of $V_k$ in \eqref{1.3}, i.e., $v_i=Q^T\tilde{v}_i$.
It can be written as
\begin{align}
& (I_{m},0_{m\times p})\widetilde{V}_{k}=U_{k+1}B_{k} \label{2.5} , \\
& QQ^{T}
\begin{pmatrix}
U_{k+1} \\
0_{p\times (k+1)}
\end{pmatrix}
=\widetilde{V}_{k}B_{k}^{T}+\alpha_{k+1}\tilde{v}_{k+1}e_{k+1}^{T} \label{2.6}  ,  \\
& (0_{p\times m},I_{p})\widetilde{V}_{k}D_k=\widehat{U}_{k}\widehat{B}_{k} \label{2.7} ,
\end{align}
where $D_k=\diag(1,-1,\dots ,(-1)^{k-1})\in\mathbb{R}^{k\times k}$,
and $e_{k+1}$ is the last column of $I_{k+1}$.

It is shown \cite{Kilmer2007} that, in exact arithmetic, the $k$-step JBD process satisfies
\begin{equation}\label{2.8}
AZ_{k} = U_{k+1}B_{k} , \ \ LZ_{k}=\widehat{U}_{k}\bar{B}_{k} ,
\end{equation}
where $Z_{k}=R^{-1}V_{k}=(z_{1},\dots,z_{k})$ and $\bar{B}_{k}=\widehat{B}_{k}D_k$,
and
\begin{equation}\label{2.9}
B_{k}^{T}B_{k}+\bar{B}_k^T\bar{B}_k=I_{k}.
\end{equation}
Therefore, the singular values of $B_{k}$
are determined by those of $\bar{B}_{k}$, i.e.,
$\widehat{B}_k$, and vice versa. As a result, if some extreme
generalized singular values of $\{A, L\}$ are of interest,
one can use the extreme generalized singular values of $\{B_k,\bar{B}_k\}$
to approximate them by either computing the generalized
singular values of the small pairs or
only computing the singular values of $B_{k}$ or
$\bar{B}_{k}$.

In finite precision, it is well known that the Lanczos vectors computed by
Lanczos bidiagonalization gradually lose their mutual orthogonality
as the iteration number $k$ increases \cite{Larsen1998}.
Following \cite{Larsen1998,Simon2000}, we define the
orthogonality level of a set of vectors as follows.

\begin{Def}\label{def2.1}
For a rectangular matrix $W_{k}=(w_{1}, \dots, w_{k})\in \mathbb{R}^{r\times k}$
with $\lVert w_{j}\lVert=1$, $j=1,\dots,k$, we call
$\xi_{ij}^{w}=|w_{i}^{T}w_{j}|$ the orthogonality level among $w_{i}$
and $w_{j}$. The orthogonality level of $\{w_{1}, \dots, w_{k}\}$ or $W_{k}$
is measured by one of
	\begin{align}
	& \xi(W_{k})=\max_{1\leq i\neq j \leq k}\xi_{ij}^{w} , \label{2.12} \\
	& \eta(W_{k}) = \| I_{k}-W_{k}^{T}W_{k}\| . \label{2.13}
	\end{align}	
\end{Def}

Notice that $\xi(W_{k})\leq\eta(W_{k})\leq k\xi(W_{k})$. The above two quantities
can be used interchangeably to measure the orthogonality level of
Lanczos vectors. Let $\sigma_{i}(\cdot)$ and $\lambda_{i}(\cdot)$ be
the $i$-th largest singular value and eigenvalue of a symmetric
matrix respectively. Then
$$\sigma_{1}^{2}(W_{k})
= \lambda_{1}(W_{k}^{T}W_{k})
= 1 + \lambda_{1}(W_{k}^{T}W_{k}-I_{k})
\leq 1 + \lVert I_{k}-W_{k}^{T}W_{k}\lVert ,$$
which leads to
\begin{equation}\label{2.14}
\|W_{k}\| \leq \sqrt{1 + \eta(W_{k})} .
\end{equation}

For Lanczos bidiagonalization, the loss of orthogonality of the Lanczos
vectors will lead to appearance of spurious Ritz values and a delay
of the convergence of Ritz values \cite{Larsen1998}. Therefore,
reorthogonalization strategies are necessary to maintain necessary orthogonality
level in order to preserve convergence of Ritz values and avoid
spurious copies;
see \cite{Barlow2013,Larsen1998,Parlett1979,Simon1984a,Simon1984b,Simon2000}
for a few types of reorthogonalization strategies and related analysis.

\section{The JBD process in finite precision}\label{sec3}

When it is carried out in finite precision, due to
the influence of rounding errors, the behavior of the JBD process
will deviate from that in exact arithmetic.
First, the JBD process of $\{A, L\}$ is not equivalent to the
combination of the two Lanczos bidiagonalizations any longer.
Second, the three matrices $U_{k+1}$, $\widetilde{V}_{k}$ and $\widehat{U}_{k}$
lose their numerical orthogonality gradually as $k$ increases. In this paper,
we do not consider the solution accuracy of the inner least
squares problem \eqref{2.3} at each iteration, though it has an
important influence on the accuracy and efficiency of the algorithm.
This issue is complicated, and we will study it in our future work.
In the following analysis, we always assume that \eqref{2.3} is solved
accurately, i.e.,
$\begin{pmatrix}
A \\ L
\end{pmatrix}\tilde{x}_i=QQ^{T} \begin{pmatrix}
u_i \\ 0_p
\end{pmatrix}$.

First of all, we state a set of basics and assumptions on
the behavior of the rounding
errors occurring in the JBD process. They are adapted from
the symmetric Lanczos process and Lanczos bidiagonalization,
constitute a model for the actual computation, and grasp essential features but
discard those irrelevant or negligible ones. From now on, without confusion,
we use the same notation as before to denote the computed ones in
finite precision. In this case,
relations \eqref{2.5}--\eqref{2.7} add
rounding error terms (cf. \cite[\S 13.4]{Parlett1980}) and become
\begin{align}
& (I_{m},0_{m\times p})\widetilde{V}_{k}=U_{k+1}B_{k} + \widetilde{F}_{k}, \label{3.1}  \\
& QQ^{T}
\begin{pmatrix}
U_{k+1} \\
0_{p\times (k+1)}
\end{pmatrix}=\widetilde{V}_{k}B_{k}^{T}+\alpha_{k+1}\tilde{v}_{k+1}
e_{k+1}^{T}+\widetilde{G}_{k+1}, \label{3.2} \\
& (0_{p\times m},I_{p})\widetilde{V}_{k}D_k=\widehat{U}_{k}
\widehat{B}_{k}+\bar{F}_{k} , \label{3.3}
\end{align}
where the rounding error matrices $ \widetilde{F}_{k}=(\tilde{f}_{1}, \dots, \tilde{f}_{k}),
\ \widetilde{G}_{k+1} = (\tilde{g}_{1}, \dots, \tilde{g}_{k+1})$
and $\bar{F}_{k}=(\bar{f}_{1}, \dots, \bar{f}_{k})$ satisfy
\begin{equation}\label{errorm}
\|\widetilde{F}_{k}\|, \|\widetilde{G}_{k+1}\|, \|\bar{F}_{k}\|
=O(\epsilon).
\end{equation}
Second, the following local orthogonality of $u_{i}$ and $\hat{u}_{i}$ holds
\cite{Paige1976,Simon1984a}:
\begin{align}
& \beta_{i+1}|u_{i+1}^{T}u_{i}| = O(c_{1}(m,n)\epsilon) , \label{3.4} \\
& \hat{\alpha}_{i+1}|\hat{u}_{i+1}^{T}\hat{u}_{i}| = O(c_{2}(p,n)\epsilon) ,\label{3.5}
\end{align}
where $c_{1}(m,n)$ and $c_{2}(p,n)$ are modest constants depending
on $m$, $n$ and $p$. Third, we assume that
\begin{equation}\label{3.6}
neither \ \ \beta_{i+1} \ nor \ \hat{\alpha}_{i+1}\  ever \  becomes \  negligible ,
\end{equation}
which is almost always true in practice. In case either
$\beta_{i+1}$ or $\hat{\alpha}_{i+1}$
becomes negligible, the JBD process
should be terminated  since it has found either acceptable approximate
right invariant generalized singular subspace and left
generalized singular subspace for $A$ or acceptable approximate
right invariant generalized singular subspace and left
generalized singular subspace for $L$, as can be justified from
\eqref{2.10}, \eqref{2.11} and  \eqref{2.8}. This can also be
seen from the later Theorem~\ref{thm4.2}
and $\hat{\alpha}_i\hat{\beta}_i=\alpha_{i+1}\beta_{i+1}$ in
step 7 of Algorithm~\ref{alg1}.
Finally, for simplicity, in our numerical analysis, we
always assume that the computed Lanczos type vectors are of unit length.

\subsection{Relationships between the JBD process and Lanczos bidiagonalization
in finite precision}

We show that, in finite precision, the JBD process of
$\{A, L\}$ is no longer equivalent to the combination of the lower and upper
Lanczos bidiagonalizations of $Q_{A}$ and $Q_{L}$. To this end, we first
present the following lemma.

\begin{Lem}\label{lem3.1}
Let $v_{i}=Q^{T}\tilde{v}_{i}$, $V_{k}=(v_{1},\dots,v_{k})$
and $\underline{B}_{k}=\begin{pmatrix}
B_{k-1}^{T} \\
\alpha_{k}e_{k}^{T}
\end{pmatrix}\in \mathbb{R}^{k\times k}$. Then
\begin{equation}\label{3.7}
\lVert \widetilde{V}_{k} - QV_{k} \lVert \leq \lVert \widetilde{G}_{k}
\underline{B}_{k}^{-1}\lVert
= O(\lVert \underline{B}_{k}^{-1}\lVert\epsilon)
\end{equation}	
with $\widetilde{G}_{k}$ defined in \eqref{3.2}.
\end{Lem}
\begin{proof}
Write the matrix $C=(A^T,L^T)^T$. Then
$$
QQ^{T}=CC^{\dag}, \ QQ^TC=C, \ QQ^{T}
\begin{pmatrix}
U_{k} \\
0_{p\times k}
\end{pmatrix}=CX_{k},
$$
where $``\dag"$ denotes the Moore--Penrose
inverse of a matrix and $X_{k}=C^{\dag}\begin{pmatrix}
U_{k} \\
0_{p\times k}
\end{pmatrix}$.
From \eqref{3.2} and $V_k=Q^T\widetilde{V}_k$,
we have $ CX_{k} = \widetilde{V}_{k}\underline{B}_{k} +
\widetilde{G}_{k}$, leading to
$\widetilde{V}_{k}=CX_{k}\underline{B}_{k}^{-1}-
\widetilde{G}_{k}\underline{B}_{k}^{-1}$.
Therefore, we obtain
\begin{eqnarray}
\widetilde{V}_{k} - QV_{k} &=& \widetilde{V}_{k}-QQ^{T}\widetilde{V}_k
=\widetilde{V}_{k}-QQ^{T}(CX_{k}\underline{B}_{k}^{-1}-
\widetilde{G}_{k}\underline{B}_{k}^{-1})\nonumber\\
&=&\widetilde{V}_{k}-CX_k\underline{B}_{k}^{-1}+
QQ^T\widetilde{G}_{k}\underline{B}_{k}^{-1}=-\widetilde{G}_k\underline{B}_k^{-1}+
QQ^T\widetilde{G}_{k}\underline{B}_{k}^{-1}
\nonumber\\
&=& (QQ^{T}-I_{m+p})\widetilde{G}_{k}\underline{B}_{k}^{-1}.\label{3.8}
\end{eqnarray}
Taking norms in both sides proves \eqref{3.7}.
\end{proof}

This lemma shows that $\widetilde{V}_{k}$ deviates from
$QV_k$ with the error $O(\lVert \underline{B}_{k}^{-1}\lVert\epsilon)$ and
$\widetilde{V}_{k}=QV_{k}$ in exact arithmetic holds no longer. Using \eqref{3.7},
we can rewrite \eqref{3.1} as
\begin{equation}\label{new}
(I_{m},0_{m\times p})Q{V}_{k}
=U_{k+1}B_{k}+ F_{k}
\end{equation}
where
\begin{equation}
F_{k}=\widetilde{F}_{k}-(I_{m},0_{m\times p})(\widetilde{V}_{k} - QV_{k}).\label{3.9}
\end{equation}
Then from \eqref{errorm} and \eqref{3.7}
we have $\|F_{k}\|=O(\lVert \underline{B}_{k}^{-1}\lVert\epsilon)$.
Premultiplying \eqref{3.2} by $Q^{T}$ and exploiting
$(I_{m},0_{m\times p})Q{V}_{k}=Q_AV_k$ straightforwardly yields
the following result,
which is the corresponding lower Lanczos bidiagonalization of
$Q_A$ in finite precision resulting from the JBD process.

\begin{theorem} \label{thm3.1}
 	Suppose that the inner least squares problem \eqref{2.3} is solved
 accurately. In finite precision, we have
	\begin{align}
	& Q_{A}V_{k} = U_{k+1}B_{k}+ F_{k} , \label{3.10} \\
	& Q_{A}^{T}U_{k+1}=V_{k}B_{k}^{T}+\alpha_{k+1}v_{k+1}e_{k+1}^{T}+G_{k+1} , \label{3.11}
	\end{align}
	where $G_{k+1}=Q^{T}\widetilde{G}_{k+1}$ with $\widetilde{G}_{k+1}$
in \eqref{3.2}
and $\lVert F_{k}\lVert= O(\lVert \underline{B}_{k}^{-1}\lVert\epsilon)$,
$\lVert G_{k+1} \lVert = O(\epsilon)$.
\end{theorem}

This theorem indicates that the error term $F_{k}$ is amplified gradually
once $\lVert \underline{B}_{k}^{-1}\lVert$ grows with $k$. When $Q_A$
is rectangular, i.e., $m>n$, theoretically we cannot control the size of $\|\underline{B}_{k}^{-1}\|$,
as shown now: from the second relation in \eqref{2.10}, we obtain
$$
U_k^TQ_AQ_A^TU_k=\underline{B}_{k}^T\underline{B}_{k},
$$
indicating that the eigenvalues of $\underline{B}_{k}^T\underline{B}_{k}$,
i.e., the squares of the singular values of $\underline{B}_{k}$,
are the Ritz values of the singular matrix $Q_AQ_A^T$ with respect to
${\rm span}(U_k)$ and lie between the largest and smallest eigenvalues
of the {\em singular} matrix $Q_AQ_A^T$. Notice
that ${\rm span}(U_k)$ is the Krylov subspace generated by
$u_1,Q_AQ_A^Tu_1,\ldots,(Q_AQ_A^T)^{k-1}u_1$. Then
that the smallest eigenvalue of $\underline{B}_{k}^T\underline{B}_{k}$
converges to the zero eigenvalue of $Q_AQ_A^T$ as $k$ increases,
so that $\|\underline{B}_{k}^{-1}\|$ may become uncontrollably large;
on the other hand, for $Q_A$ flat or square, i.e., $m\leq n$,
and having full row rank, however, such a phenomenon definitively
cannot occur, and the smallest eigenvalue
of $\underline{B}_{k}^T\underline{B}_{k}$
is bounded from below by the smallest positive one of $Q_AQ_A^T$. In this case,
$\|\underline{B}_{k}^{-1}\|$ is always uniformly bounded but possibly large
when $Q_A$ is ill conditioned. We refer the reader to \cite{jia2020b} on
a detailed analysis on $\underline{B}_k$ and its singular values,
which is closely related to and play a critical role in
the conjugate gradient minimal error (CGME) method
for solving least squares problems and linear discrete ill-posed problems.
As a result, in finite precision,
the the JBD process for computing $U_{k+1}$, $V_k$ and $B_k$ may be
no longer equivalent to
the lower Lanczos bidiagonalization of $Q_{A}$, where the rounding error
term in the place of $F_k$ is
$O(\|Q_{A}\|\epsilon)=O(\epsilon)$ in size.

Similarly, from \eqref{1.5} and the first relation in \eqref{2.11} we have
$$
V_k^TQ_L^TQ_L V_k=\bar{B}_k^T\bar{B}_k.
$$
Since ${\rm span}(V_k)$ is the Krylov subspace generated by
$v_1,Q_L^TQ_Lv_1,\ldots,(Q_L^TQ_L)^{k-1}v_1$,
we can make a similar analysis to the above. Specifically,
if $Q_L$ is rectangular or square and of full column rank, then
the smallest singular value
of $\bar{B}_k$ converges to the smallest one of $Q_L$ from above
as $k$ increases, and it is bounded from below by it, meaning
that $\|\bar{B}_k^{-1}\|$ is controllable; if $Q_L$ is flat, then
the smallest singular value of $\bar{B}_k$ converges to zero as
$k$ increases, causing that $\|\bar{B}_k^{-1}\|$ cannot be controlled
and become large as $k$ increases.

Note that the orders of $Q_A$ and $Q_L$ are the same as those of
$A$ and $L$, respectively.
The above analysis and assertions suggest us to first check
the orders of $A$ and $L$ and then perform the JBD process on either
$\{A,L\}$ or $\{L,A\}$ when attempting to ensure that the resulting
$\lVert \underline{B}_{k}^{-1}\lVert$
and $\|\bar{B}_k^{-1}\|$ are bounded whenever possible. As will be
seen later, their boundedness is desirable for the JBD process
and the JBD method for the GSVD computation in finite precision.

In finite precision, we next prove that relation \eqref{2.9} holds within $O(\epsilon)$,
which is important to design an efficient and reliable algorithm for the GSVD
computation. To this end, we first establish upper bounds for $\|B_{k}\|$
and $\|\widehat{B}_{k}\|$. From \eqref{3.1} and \eqref{3.3}, at the $i$-th step
we have
\begin{align}
& \tilde{v}_{i}(1:m)=\alpha_{i}u_{i}+\beta_{i+1}u_{i+1}+\tilde{f}_i , \label{3.12} \\
& (-1)^{i-1}\tilde{v}_{i}(m+1:m+p)= \hat{\alpha}_{i}\hat{u}_{i}
+\hat{\beta}_{i-1}\hat{u}_{i-1}+\bar{f}_{i} . \label{3.13}
\end{align}
Thus $\| \alpha_{i}u_{i}+\beta_{i+1}u_{i+1}\|^{2} = \| \tilde{v}_{i}(1:m)-\tilde{f}_i\|^{2}$,
which leads to
\begin{align}\label{3.14}
\begin{split}
\alpha_{i}^{2}+\beta_{i+1}^2
&=\|\tilde{v}_{i}(1:m)\|^{2}+\| \tilde{f}_i\|^{2}-
2\tilde{f}_i^{T}\tilde{v}_{i}(1:m)-2\alpha_{i}\beta_{i+1}u_{i+1}^{T}u_{i} \\
&\leq 1+O(c_{1}(m,n)\epsilon),
\end{split}
\end{align}
where we have used \eqref{3.4}. Similarly, we obtain
\begin{equation}\label{3.15}
\hat{\alpha}_{i}^2 + \hat{\beta}_{i-1}^{2}\leq 1+O(c_{2}(p,n)\epsilon) .
\end{equation}
Therefore, we have established the following lemma.
\begin{Lem}
In finite precision, we have\footnote{Here we use the result of an exercise
from \cite[Chapter 6, Problem 6.14]{Higham2002}, which gives an upper
bound for the $p$-norm of a row/column sparse matrix.}
\begin{align}
& \|B_{k}\| \leq \sqrt{2}\max_{1\leq i \leq k}(\alpha_{i}^{2}+\beta_{i+1}^{2})^{1/2}
\leq \sqrt{2}+O(c_{1}(m,n)\epsilon), \label{3.16} \\
& \|\bar{B}_{k}\| \leq \sqrt{2}\max_{1\leq i \leq k}(\hat{\alpha}_{i}^{2}+\hat{\beta}_{i-1}^{2})^{1/2}
\leq \sqrt{2}+O(c_{2}(p,n)\epsilon) \label{3.17} .
\end{align}
\end{Lem}

We are ready to derive an important relationship between $B_{k}$
and $\widehat{B}_{k}$ in finite precision, which was
presented in \cite[Theorem 3.1]{JiaLi2021}. Because of
its importance, here we give a simpler proof for the completeness.

\begin{theorem}\label{thm3.2}
With the hypothesis of Theorem~\ref{thm3.1}, in finite precision we have
\begin{equation}\label{3.18}
B_{k}^{T}B_{k}+\bar{B}_k^T\bar{B}_k=I_{k} + E_{k } ,
\end{equation}
where $\bar{B}_k=\widehat{B}_kD_k$ with $D_k=\diag(1,-1,\dots ,(-1)^{k-1})$,
and $E_{k}$ is symmetric tridiagonal
with its nonzero elements being $O(c_{3}(m,n,p) \epsilon)$ and
$c_{3}(m,n,p) =c_{1}(m,n)+c_{2}(p,n)$ in size.
\end{theorem}

\begin{proof}
Since
$$B_{k}^{T}B_{k}=\begin{pmatrix}
\alpha_{1}^{2}+\beta_{2}^{2} &\alpha_{2}\beta_{2} & & \\
\alpha_{2}\beta_{2} &\alpha_{2}^{2}+\beta_{3}^{2} &\ddots & \\
&\ddots &\ddots & \alpha_{k}\beta_{k}\\
& &\alpha_{k}\beta_{k} &\alpha_{k}^{2}+\beta_{k+1}^{2}
\end{pmatrix},$$
$$
\widehat{B}_{k}^{T}\widehat{B}_{k}=\begin{pmatrix}
\hat{\alpha}_{1}^{2}&\hat{\alpha}_{1}\hat{\beta}_{1} & & \\
\hat{\alpha}_{1}\hat{\beta}_{1}&\hat{\alpha}_{2}^{2}+\hat{\beta}_{1}^{2} &\ddots & \\
&\ddots &\ddots &\hat{\alpha}_{k-1}\hat{\beta}_{k-1} \\
& &\hat{\alpha}_{k-1}\hat{\beta}_{k-1} &\hat{\alpha}_{k}^{2}+\hat{\beta}_{k-1}^{2}
\end{pmatrix}
$$
are symmetric tridiagonal, $\bar{B}_k^T\bar{B}_k$ and $E_k$ in \eqref{3.18} are too.	

For the diagonal part, in finite precision, we have
$$
\hat{\beta}_{i}=(\alpha_{i+1}\beta_{i+1}/\hat{\alpha}_{i})
(1+\rho_i),
$$
where $|\rho_i|= O(\epsilon)$
\cite[Lemma 3.1]{Higham2002}, leading to
$$
\alpha_{i+1}\beta_{i+1} =\hat{\alpha}_{i}\hat{\beta}_{i}-\alpha_{i+1}\beta_{i+1}\rho_i.
$$
From \eqref{3.14} we have
$$\alpha_{i+1}\beta_{i+1} \leq \frac{\alpha_{i+1}^{2}+\beta_{i+1}^2}{2}
\leq \frac{2[1+O(c_{1}(m,n)\epsilon)]}{2}
=1+O(c_{1}(m,n)\epsilon).$$
Therefore, we obtain
\begin{equation}\label{3.19}
\alpha_{i+1}\beta_{i+1}=\hat{\alpha}_{i}\hat{\beta}_{i} +\gamma_{i}  ,
\end{equation}
where $|\gamma_{i}| \leq [1+O(c_{1}(m,n)\epsilon)]\epsilon = O(\epsilon)$.

For the subdiagonal part, by taking norms in \eqref{3.12} and \eqref{3.13},
we have
\begin{align*}
& \ \ \ \ \|\tilde{v}_{i}(1:m) \|^{2}+\|\tilde{v}_{i}(m+1:m+p)\|^{2} \\
&= \|\alpha_{i}u_{i}+\beta_{i+1}u_{i+1}+\tilde{f}_i\|^{2}+\|\hat{\alpha}_{i}\hat{u}_{i}
+\hat{\beta}_{i-1}\hat{u}_{i-1}+\bar{f}_{i}\|^{2} \\
&= \alpha_{i}^{2}+\beta_{i+1}^{2}+2\alpha_{i}\beta_{i+1}u_{i}^{T}u_{i+1}
+2\alpha_{i}u_{i}^{T}\tilde{f}_{i}+2\beta_{i+1}u_{i+1}^{T}
\tilde{f}_{i}+\|\tilde{f}_{i}\|^{2} \\
&\ \ \ \ +\hat{\alpha}_{i}^{2}+\hat{\beta}_{i-1}^{2}
+2\hat{\alpha}_{i}\hat{\beta}_{i-1}\hat{u}_{i}^{T}\hat{u}_{i-1}
+2\hat{\alpha}_{i}\hat{u}_{i}^{T}\bar{f}_{i}
+2\hat{\beta}_{i-1}\hat{u}_{i-1}^{T}\bar{f}_{i}+\|\bar{f}_{i}\|^{2} .
\end{align*}
From \eqref{3.14} and \eqref{3.15} we obtain
\begin{eqnarray*}
\alpha_{i}+\beta_{i+1}& \leq &\sqrt{2(\alpha_{i}^{2}+\beta_{i+1}^2)}
\leq\sqrt{2}+O(c_{1}(m,n)\epsilon),\\
\hat{\alpha}_{i}+\hat{\beta}_{i-1} &\leq& \sqrt{2(\hat{\alpha}_{i}^{2}+
\hat{\beta}_{i-1}^{2})}\leq\sqrt{2}+O(c_{2}(p,n)\epsilon),
\end{eqnarray*}
showing that
\begin{align*}
& \ \ \ \big|2\alpha_{i}u_{i}^{T}\tilde{f}_{i}+
2\beta_{i+1}u_{i+1}^{T}\tilde{f}_{i}+\|\tilde{f}_{i}\|^{2} +
2\hat{\alpha}_{i}\hat{u}_{i}^{T}\bar{f}_{i}
+2\hat{\beta}_{i-1}\hat{u}_{i-1}^{T}\bar{f}_{i}+\|\bar{f}_{i}\|^{2} \big| \\
& = O(2(\alpha_{i} + \beta_{i+1})\epsilon)+
O(2(\hat{\alpha}_{i} + \hat{\beta}_{i-1})\epsilon) = O(\epsilon) ,
\end{align*}
where we have neglected the higher order term $O(\epsilon^2)$.
Exploiting \eqref{3.4} and \eqref{3.5}, we obtain
$$
|2\alpha_{i}\beta_{i+1}u_{i}^{T}u_{i+1}+ 2\hat{\alpha}_{i}
\hat{\beta}_{i-1}\hat{u}_{i}^{T}\hat{u}_{i-1}|
= O(c_{3}(m,n,p) \epsilon) $$
with $c_{3}(m,n,p)=c_{1}(m,n)+c_{2}(p,n)$. Since
$$
1 = \|\tilde{v}_{i}\|^{2} = \|\tilde{v}_{i}(1:m) \|^{2}+\|\tilde{v}_{i}(m+1:m+p)\|^{2},
$$
we have
\begin{equation}\label{3.20}
\alpha_{i}^{2}+\beta_{i+1}^{2}+\hat{\alpha}_{i}^{2}+\hat{\beta}_{i-1}^{2}
= 1 + O(c_{3}(m,n,p) \epsilon) .
\end{equation}
Combining \eqref{3.19} with \eqref{3.20} gives rise to \eqref{3.18}.
\end{proof}

It is remarkable to notice that \eqref{3.18} holds without depending on
the orthogonality levels of $U_{k+1},\ V_k$ and $\widetilde{V}_k$.
Since \eqref{3.18} indicates that the squares of singular values
of $B_k$ and $\bar{B}_k$ can be determined each other with the error
$O(\epsilon)$, it can be used to design algorithms for the GSVD
computation. Therefore, if some
extreme generalized singular values of $\{A, L\}$ are desired, then, except
the approach that computes the generalized singular values of $\{B_k,\bar{B}_k\}$
and uses extreme ones to approximate the corresponding ones of
$\{A,L\}$, an alternative is to compute the singular values of
$B_{k}$ or $\widehat{B}_{k}$ (or equivalently $\bar{B}_k$) to obtain
approximate generalized singular values. For a convergence analysis of
the singular values of $B_k$ to some of $Q_A$ and more details,
we refer the reader to \cite{jia2020,jia2020b}. We will propose
the JBD method for the GSVD computation and consider some details
in Section \ref{sec4}, where we will revisit \eqref{3.18}.

Since $B_{k}$ and $\bar{B}_{k}$ have full column rank, by \eqref{3.18} we have
\begin{align*}
(B_{k}^{T})^{\dag}(B_{k}^{T}B_{k}+\bar{B}_{k}^{T}\bar{B}_{k})\bar{B}_{k}^{-1}
&= (B_{k}^{T})^{\dag}\bar{B}_{k}^{-1}+(B_{k}^{T})^{\dag}E_{k}\bar{B}_{k}^{-1} \\
&= (\bar{B}_{k}B_{k}^{T})^{\dag}+(B_{k}^{T})^{\dag}E_{k}\bar{B}_{k}^{-1},
\end{align*}
which leads to
$$ (\bar{B}_{k}B_{k}^{T})^{\dag} = (B_{k}^{T})^{\dag}B_{k}^{T}B_{k}
\bar{B}_{k}^{-1}+(B_{k}^{T})^{\dag}\bar{B}_{k}^{T}
- (B_{k}^{T})^{\dag}E_{k}\bar{B}_{k}^{-1} . $$
Since $B_{k}^{T}=\begin{pmatrix}
\underline{B}_{k}, \beta_{k+1}e_{k}
\end{pmatrix}$,
we have $\|(B_{k}^{T})^{\dag}\|\leq \|\underline{B}_{k}^{-1}\|$. Notice that
$\|(B_{k}^{T})^{\dag}B_{k}^{T}\|=1$. From \eqref{3.16} and \eqref{3.17}, we obtain
$$\|(\bar{B}_{k}B_{k}^{T})^{\dag} \| \leq \sqrt{2}(\|\bar{B}_{k}^{-1}\|+
\|\underline{B}_{k}^{-1}\|)+c_{0}(\epsilon) ,$$
where $$c_{0}(\epsilon)=
\|\bar{B}_{k}^{-1}\|\|\underline{B}_{k}^{-1}\|O(c_{3}(m,n,p) \epsilon) .
$$

Since $\bar{B}_{k}B_{k}^{T}=\begin{pmatrix}
\bar{B}_{k}\underline{B}_{k}, \beta_{k+1}\bar{B}_{k}e_{k}
\end{pmatrix}$, by the interlacing property of singular values, we have
$\sigma_{k}(\bar{B}_{k}\underline{B}_{k})\leq \sigma_{k}(\bar{B}_{k}B_{k}^{T})$,
where $\sigma_{i}(\cdot)$ denotes the $i$-th largest singular value of a matrix,
proving that
$\|(\bar{B}_{k}B_{k}^{T})^{\dag} \| \leq \| (\bar{B}_{k}\underline{B}_{k})^{-1}\|$.
However, in our experiments, it has been found that
$ \| (\bar{B}_{k}\underline{B}_{k})^{-1}\|\leq \bar{c}\|(\bar{B}_{k}B_{k}^{T})^{\dag} \|$
always holds with a modest constant $\bar{c}=O(1)$, usually $\bar{c} \leq 5$.
Notice that $c_{0}(\epsilon)$ is much smaller than
$\|\bar{B}_{k}^{-1}\|+\|\underline{B}_{k}^{-1}\|$. Therefore,
in the sequel we
will suppose and use the following upper bound
for $\|(\bar{B}_{k}\underline{B}_{k})^{-1}\|$:
\begin{equation}\label{3.21}
\|(\bar{B}_{k}\underline{B}_{k})^{-1}\|\leq c\sqrt{2}(\|\bar{B}_{k}^{-1}\|
+\|\underline{B}_{k}^{-1}\|),
\end{equation}
where $c> 1$ slightly.

We next establish a relationship between the JBD process for
computing $\widehat{B}_{k}$ and the upper Lanczos bidiagonalization of $Q_{L}$
in finite precision.

\begin{theorem}\label{thm3.3}
	With the hypothesis of Theorem \ref{thm3.1}, in finite precision, we have
\begin{eqnarray}
Q_{L}\widehat{V}_{k}&=&\widehat{U}_{k}\widehat{B}_{k}+\widehat{F}_{k},\label{3.23}\\
	Q_{L}^{T}\widehat{U}_{k}&=&\widehat{V}_{k}\widehat{B}_{k}^{T}
+\hat{\beta}_{k}\hat{v}_{k+1}e_{k}^{T}+\widehat{G}_{k}\label{3.24}
	\end{eqnarray}
	with
	\begin{eqnarray}
 \|\widehat{F}_{k} \|&=&O(\lVert \underline{B}_{k}^{-1}\lVert\epsilon),\\
	\|\widehat{G}_{k}\|&=& O(c_{4}(m,n,p,k)\epsilon)\label{3.25}
	\end{eqnarray}
with
\begin{equation}\label{c4}
c_{4}(m,n,p,k)=\|\underline{B}_{k}^{-1}\|+c_{3}(m,n,p)\|\widehat{B}_{k}^{-1}\|.
\end{equation}
\end{theorem}

\begin{proof}
Recall the notation in \eqref{1.5}, \eqref{2.7} and \eqref{2.11}.
Then exploiting Lemma~\ref{lem3.1}, we can rewrite \eqref{3.3} as
$$(0_{p\times m},I_{p})\widehat{V}_{k}=\widehat{U}_{k}\widehat{B}_{k}+\widehat{F}_{k},$$
where
\begin{equation}\label{3.22}
\widehat{F}_{k}=\bar{F}_{k}-(0_{p\times m},I_{p})(\widetilde{V}_{k}-QV_{k})D_k.
\end{equation}
Therefore, \eqref{3.23} holds.

From \eqref{3.10} and \eqref{3.11}, we obtain
\begin{align}\label{3.26}
\begin{split}
Q_{A}^{T}Q_{A}V_{k} &=Q_{A}^{T}U_{k+1}B_{k}+Q_{A}^{T}F_{k} \\
&= (V_{k}B_{k}^{T}+\alpha_{k+1}v_{k+1}e_{k+1}^{T}+G_{k+1})B_{k}+Q_{A}^{T}F_{k} \\
&= V_{k}B_{k}^{T}B_{k}+\alpha_{k+1}\beta_{k+1}v_{k+1}e_{k}^{T}+G_{k+1}B_{k}+
Q_{A}^{T}F_{k} .
\end{split}
\end{align}
Premultiplying and postmultiplying \eqref{3.23} by $Q_{L}^{T}$ and $D_k$ gives
$$
Q_{L}^{T}Q_{L}V_{k}=(Q_{L}^{T}\widehat{U}_{k}\widehat{B}_{k}+
Q_{L}^{T}\widehat{F}_{k})D_k.
$$
Summing the above two equalities yields
\begin{align*}
V_{k} &=(Q_{A}^{T}Q_{A}+Q_{L}^{T}Q_{L})V_{k} \\
&= V_{k}B_{k}^{T}B_{k}+Q_{L}^{T}\widehat{U}_{k}\widehat{B}_{k}D_k+\alpha_{k+1}
\beta_{k+1}v_{k+1}
e_{k}^{T} +(G_{k+1}B_{k}+Q_{A}^{T}F_{k}+Q_{L}^{T}\widehat{F}_{k}D_k) \\
&=  V_{k}(I_{k}-D_k\widehat{B}_{k}^{T}\widehat{B}_{k}D_k+E_{k})
+Q_{L}^{T}\widehat{U}_{k}\widehat{B}_{k}D_k+\alpha_{k+1}\beta_{k+1}v_{k+1}e_{k}^{T} \\
&\ \ \ \ +(G_{k+1}B_{k}+Q_{A}^{T}F_{k}+Q_{L}^{T}\widehat{F}_{k}D_k).
\end{align*}	
From this relation, $\widehat{V}_k=V_kD_k$ and $D_k^2=I_k$ it follows that
\begin{align*}
\widehat{V}_{k}\widehat{B}_{k}^{T}\widehat{B}_{k}
&= Q_{L}^{T}\widehat{U}_{k}\widehat{B}_{k}+\alpha_{k+1}\beta_{k+1}v_{k+1}e_{k}^{T}D_k+
(G_{k+1}B_{k}+Q_{A}^{T}F_{k} +Q_{L}^{T}\widehat{F}_{k}D_k+V_{k}E_{k})D_k \\
&= Q_{L}^{T}\widehat{U}_{k}\widehat{B}_{k}-(\hat{\alpha}_{k}\hat{\beta}_{k}+
\gamma_{k})\hat{v}_{k+1}
e_{k}^{T}+(G_{k+1}B_{k}+Q_{A}^{T}F_{k}+Q_{L}^{T}\widehat{F}_{k}D_k+V_{k}E_{k})D_k,
\end{align*}
which shows that
$$\widehat{V}_{k}\widehat{B}_{k}^{T}
=Q_{L}^{T}\widehat{U}_{k}-\hat{\beta}_{k}\hat{v}_{k+1}e_{k}^{T}+ E_{1} + E_{2},$$
where
$$E_{1} = [(G_{k+1}B_{k}+V_{k}E_{k})D_k- \gamma_{k}\hat{v}_{k+1}e_{k}^{T}]\widehat{B}_{k}^{-1} , \ \
E_{2} = (Q_{A}^{T}F_{k}D_k + Q_{L}^{T}\widehat{F}_{k})\widehat{B}_{k}^{-1} .$$

From \eqref{2.14}, we have $\|V_{k}\| \leq \sqrt{1+\eta(V_{k})}$.
Notice that $\gamma_{k}$ in \eqref{3.19} satisfies $|\gamma_{k}|=O(\epsilon)$
and the elements of $\|E_{k}\|$ in \eqref{3.18} are $O(c_{3}(m,n,p))$.
Using the upper bounds for $\|B_{k}\|$ in \eqref{3.16}, we have
$$\|E_{1}\|=O(\bar{c}_{1}(m,n,p,k)\epsilon) $$ with $\bar{c}_{1}(m,n,p,k)
=(\sqrt{2}+c_{3}(m,n,p))\|\widehat{B}_{k}^{-1}\|$. Using
the expressions of $F_{k}$, $\widehat{F}_{k}$ and
$\widetilde{V}_{k} - QV_{k}$ in \eqref{3.9}, \eqref{3.22} and \eqref{3.8},
respectively, we have
\begin{align*}
Q_{A}^{T}F_{k}D_k + Q_{L}^{T}\widehat{F}_{k}
&= \begin{pmatrix}
Q_{A}^{T} & Q_{L}^{T}
\end{pmatrix}
\begin{pmatrix}
F_{k}D_k \\ \widehat{F}_{k}
\end{pmatrix} \\
&= Q^{T}\Big[
\begin{pmatrix}
\widetilde{F}_{k}D_k \\ \bar{F}_{k}
\end{pmatrix} -
\begin{pmatrix}
I_{m} & 0_{m\times p} \\
0_{p\times m} & I_{p}
\end{pmatrix}
(\widetilde{V}_{k}-QV_{k})D_k
 \Big] \\
&= Q^{T}\Big[
\begin{pmatrix}
\widetilde{F}_{k}D_k \\ \bar{F}_{k}
\end{pmatrix} + (I_{m+p}-QQ^{T})\widetilde{G}_{k}\underline{B}_{k}^{-1}D_k \Big] .
\end{align*}
From \eqref{3.21}, we obtain
$$\|\underline{B}_{k}^{-1}D_k\widehat{B}_{k}^{-1}\|=\|(\bar{B}_{k}
\underline{B}_{k})^{-1}\|\leq c\sqrt{2}(\|\bar{B}_{k}^{-1}\|+\|\underline{B}_{k}^{-1}\|). $$
Since $\|\bar{B}_{k}^{-1}\|=\|\widehat{B}_{k}^{-1}\|$, it holds that
$$\|E_{2}\|=O(\bar{c}_{2}(k)\epsilon) $$
with $\bar{c}_{2}(k)=c\sqrt{2}(\|\widehat{B}_{k}^{-1}\|+\|\underline{B}_{k}^{-1}\|)
+\|\widehat{B}_{k}^{-1}\|$. Letting $\widehat{G}_{k}=-E_{1}-E_{2}$ leads to
the desired result.
\end{proof}

Notice that \eqref{3.23} and \eqref{3.24} are the corresponding versions
of \eqref{2.11} in the presence of rounding errors. Therefore,
this theorem indicates that the JBD process for computing $\widehat{U}_{k}$,
$\widehat{V}_{k}$ and $\widehat{B}_{k}$ is no longer equivalent to
the standard upper Lanczos bidiagonalization of $Q_{L}$ with the rounding error
$O(\|Q_{L}\|\epsilon)=O(\epsilon)$ in finite precision,
since the error terms $\widehat{F}_{k}$ and $\widehat{G}_{k}$
are amplified gradually
once $\|\underline{B}_{k}^{-1}\|$ and $\|\widehat{B}_{k}^{-1}\|$ grow
as $k$ increases.

In summary, unlike standard lower and upper Lanczos bidiagonalizations
of $Q_A$ and $Q_L$ in finite precision,
the JBD process in finite precision is lower and
upper Lanczos bidiagonalizations with gradually growing error terms,
where the growths of the error terms are dictated by
$\|\underline{B}_{k}^{-1}\|$ and $\|\widehat{B}_{k}^{-1}\|$ separately.


\subsection{Loss of orthogonality of the Lanczos vectors}

It is well known that, for Lanczos bidiagonalization, the
numerical orthogonality of Lanczos vectors is gradually lost due to
the influence of rounding errors. Once the orthogonality
is lost at one iteration, the errors will propagate to later steps,
which leads to the more loss of orthogonality of subsequently computed
Lanczos vectors \cite{Larsen1998,Simon1984a}. For the JBD process,
there is similar loss of orthogonality of each of the three sets of basis vectors.
As it will turn out, the orthogonality levels of
$U_{k+1}$, $\widetilde{V}_{k}$ and $\widehat{U}_{k}$ are closely related
and interact.
Below we prove how the orthogonality level of $\widehat{U}_{k}$ is affected
by those of both $U_{k+1}$ and $\widetilde{V}_{k}$.

\begin{theorem}\label{thm3.4}
With the hypothesis of Theorem \ref{thm3.1}, in finite precision, we have
\begin{equation}\label{3.27}
\eta(\widehat{U}_{k}) \leq \|\widehat{B}_{k}^{-1}\|^{2}\big[\eta(\widetilde{V}_{k})+2\eta(U_{k+1})
+O(c_{3}(m,n,p)\epsilon)\big] .
\end{equation}
\end{theorem}

\begin{proof}
From \eqref{3.1} and \eqref{3.3}, we have
$$\widetilde{V}_{k}=
\begin{pmatrix}
U_{k+1}B_{k} \\\widehat{U}_{k}\bar{B}_{k}
\end{pmatrix} +
\begin{pmatrix}
\widetilde{F}_{k} \\ \bar{F}_{k}D_k
\end{pmatrix} ,$$
which shows that
\begin{equation}\label{vkuk}
\widetilde{V}_{k}^{T}\widetilde{V}_{k}=B_{k}^{T}U_{k+1}^{T}U_{k+1}B_{k}+
\bar{B}_{k}^{T}\widehat{U}_{k}^{T}\widehat{U}_{k}\bar{B}_{k}+E_{3},
\end{equation}
where
$$ E_{3}= B_{k}^{T}U_{k+1}^{T}\widetilde{F}_{k}+\bar{B}_{k}^{T}\widehat{U}_{k}^{T}\bar{F}_{k}P
+\widetilde{F}_{k}^{T}U_{k+1}B_{k}+P\bar{F}_{k}^{T}\widehat{U}_{k}\bar{B}_{k}
+\widetilde{F}_{k}^{T}\widetilde{F}_{k}+P\bar{F}_{k}^{T}\bar{F}_{k}D_k. $$
From \eqref{3.18}, we obtain
$$I_{k}-\widetilde{V}_{k}^{T}\widetilde{V}_{k}= B_{k}^{T}(I_{k+1}-U_{k+1}^{T}U_{k+1})B_{k}+
\bar{B}_{k}^{T}(I_{k}-\widehat{U}_{k}^{T}\widehat{U}_{k})\bar{B}_{k}-E_{k}-E_{3},
$$
which, together with \eqref{vkuk}, yields
\begin{equation}\label{3.28}
I_{k}-\widehat{U}_{k}^{T}\widehat{U}_{k}= \bar{B}_{k}^{-T}
\big[(I_{k}-\widetilde{V}_{k}^{T}\widetilde{V}_{k})-B_{k}^{T}(I_{k+1}-
U_{k+1}^{T}U_{k+1})B_{k}+E_{k}+E_{3} \big]\bar{B}_{k}^{-1} .
\end{equation}

Notice that $\|U_{k+1}\|\leq(1+\eta(U_{k+1}))^{1/2}$ and
$\|\widehat{U}_{k}\|\leq(1+\eta(\widehat{U}_{k}))^{1/2}$.
Using the bounds for $\|B_{k}\|$ and $\|\widehat{B}_{k}\|$ in \eqref{3.16} and \eqref{3.17},
respectively, by a simple calculation, we obtain
$$\|E_{3}\|=O(\epsilon) .$$
Using the bound for $\|B_{k}\|$, we have
\begin{align*}
\|B_{k}^{T}(I_{k+1}-U_{k+1}^{T}U_{k+1})B_{k}\|
&\leq \|B_{k}\|^{2}\|I_{k+1}-U_{k+1}^{T}U_{k+1}\| \\
&\leq 2\|I_{k+1}-U_{k+1}^{T}U_{k+1}\|+O(c_{1}(m,n)\epsilon) .
\end{align*}
Taking norms in \eqref{3.28} proves the desired result.
\end{proof}

The above theorem indicates that as long as $\widehat{B}_{k}$ is not
ill conditioned, the orthogonality of $\widehat{U}_{k}$ is as good as
those of $U_{k+1}$ and $\widetilde{V}_{k}$. Therefore, it is only necessary
to perform some sort of reorthogonalization strategies to maintain
desired orthogonality levels of $U_{k+1}$ and $\widetilde{V}_{k}$.

The following result shows that the orthogonality levels
of the long $\widetilde{V}_k\in\mathbb{R}^{(m+p)\times k}$ and
the short $V_k\in\mathbb{R}^{p\times k}$ are the same within $O(\epsilon)$ under
some mild condition.

\begin{theorem}
With definition \eqref{2.13}, it holds that
\begin{equation}\label{3.29}
|\eta(\widetilde{V}_{k})-\eta(V_{k})|= O(\lVert \underline{B}_{k}^{-1}\lVert^{2}\epsilon^{2}) .
\end{equation}
\end{theorem}
\begin{proof}
From \eqref{3.8}, we have
\begin{align*}
\widetilde{V}_{k}^{T}(I_{m+p}-QQ^{T})\widetilde{V}_{k}
&= \widetilde{V}_{k}^{T}(\widetilde{V}_{k} - QV_{k})
= \widetilde{V}_{k}^{T}(QQ^{T}-I_{m+p})\widetilde{G}_{k}\underline{B}_{k}^{-1}\\
&=-[(I_{m+p}-QQ^{T})\widetilde{V}_{k}]^{T}\widetilde{G}_{k}\underline{B}_{k}^{-1} \\
&=  \underline{B}_{k}^{-T}\widetilde{G}_{k}^{T}(I_{m+p}-QQ^{T})\widetilde{G}_{k}\underline{B}_{k}^{-1} .
\end{align*}
Thus
\begin{align*}
\begin{split}
I_{k} - V_{k}^{T}V_{k}
&= I_{k} - \widetilde{V}_{k}^{T}QQ^{T}\widetilde{V}_{k} =
I_{k} - \widetilde{V}_{k}^{T}\widetilde{V}_{k} +
\widetilde{V}_{k}^{T}(I_{m+p}-QQ^{T})\widetilde{V}_{k} \\
&= (I_{k} - \widetilde{V}_{k}^{T}\widetilde{V}_{k}) +
\underline{B}_{k}^{-T}\widetilde{G}_{k}^{T}(I_{m+p}-QQ^{T})
\widetilde{G}_{k}\underline{B}_{k}^{-1}.
\end{split}
\end{align*}
Therefore, \eqref{3.29} holds.
\end{proof}

This theorem shows that provided that $\|\underline{B}_{k}^{-1}\|$ is not too big, say no more
than $\epsilon^{-1/2}$, we have $|\eta(\widetilde{V}_{k})-\eta(V_{k})|=O(\epsilon)$.
Therefore, it is only necessary to maintain the same
orthogonality level of $V_k$ in order to make $\widetilde{V}_k$ achieve
a desired orthogonality level. From \eqref{3.27}, this shows that
we only need to maintain desired orthogonality levels of $U_{k+1}$ and $V_{k}$ so as to
make $\widehat{U}_k$ have a desired orthogonality level.

\section{The JBD method for the GSVD computation}\label{sec4}

In this section, we review how to use the JBD process to compute some extreme
GSVD components of $\{A,L\}$ \cite{Zha1996}, which leads to the JBD method for
the large GSVD computation, and make an analysis on
the convergence of the approximate generalized singular
values in finite precision by exploiting the previous
results.

\subsection{The JBD method}\label{sec4.1}

For ease of presentation, we do not take into account rounding errors
when computing the GSVD of $\{B_k,\bar{B}_k\}$ or the SVD of
$B_k$ or $\widehat{B}_k$, that is, we assume that the compact SVD of $B_{k}$
is computed accurately:
\begin{equation}\label{4.1}
B_{k} = P_{k}\Theta_{k}W_{k}^{T}, \ \ \Theta_{k}=\diag(c_{1}^{(k)}, \dots,
c_{k}^{(k)}), \ \
1 \geq c_{1}^{(k)} > \dots > c_{k}^{(k)} \geq 0,
\end{equation}
where $P_{k}=(p_{1}^{(k)}, \dots, p_{k}^{(k)})\in\mathbb{R}^{(k+1)\times k}$
is orthonormal
and $W_{k}=(w_{1}^{(k)}, \dots, w_{k}^{(k)})\in\mathbb{R}^{k\times k}$
is orthogonal. The SVD \eqref{4.1} can be obtained by a standard SVD
algorithm since $B_{k}$ is small sized. Then we have $k$ approximate
generalized singular values
$\{c_{i}^{(k)}, (1-(c_{i}^{(k)})^{2})^{1/2}\}$, $i=1,2,\ldots,k$ of $\{A, L\}$,
and the approximate right generalized singular vectors
are the $x_{i}^{(k)}=R^{-1}V_{k}w_{i}^{(k)}$ and the approximations to
the left generalized singular vectors $p_i^A$
are $y_{i}^{(k)} = U_{k+1}p_{i}^{(k)}$. Among these $k$ approximations, we
pick up a few largest and/or smallest ones as approximations to the
largest and/or smallest $c_i/s_i$ and the corresponding $x_i$ and $p_i^A$.

If we also want to compute an approximation of the left generalized singular
vector $p_i^L$, we need to compute the SVD of $\bar{B}_{k}$.
From \eqref{2.9}, it is known that $(B_k^T,\bar{B}_k^T)^T$ is column
orthonormal. Therefore, the CS decomposition of the pair
$\{B_k,\bar{B}_k\}$ is its GSVD, and
that the right singular vectors of $B_k$ and $\bar{B}_k$ are identical. As
a result, we can assume that the SVD of $\bar{B}_{k}$ is
\begin{equation}\label{4.2}
\bar{B}_{k} = \bar{P}_{k}\Psi_{k}W_{k}^{T}, \ \
\Psi_{k}={\rm diag}(\bar{s}_{1}^{(k)}, \dots, \bar{s}_{k}^{(k)}), \ \
0 \leq \bar{s}_{1}^{(k)} < \dots < \bar{s}_{k}^{(k)} \leq 1 \ ,
\end{equation}
where $\bar{P}_{k}=(\bar{p}_{1}^{(k)}, \dots, \bar{p}_{k}^{(k)})
\in\mathbb{R}^{k\times k}$ and $W_{k}=(w_{1}^{(k)}, \dots,
w_{k}^{(k)})\in\mathbb{R}^{k\times k}$ are orthogonal.
Then $z_{i}^{(k)}= \widehat{U}_{k}\bar{p}_{i}^{(k)}$
is an approximation to $p_{i}^{L}$. The approximate generalized singular values
and the corresponding approximate right generalized singular vectors are
$\{(1-(\bar{s}_{i}^{(k)})^{2})^{1/2},
\bar{s}_{i}^{(k)}\}$ and $R^{-1}V_{k}w_{i}^{(k)}$, respectively.
We have seen that the approximate right generalized singular vectors
obtained by the SVDs of $B_k$ and $\bar{B}_k$ are the same, as expected.
We also comment that if the JBD process is performed in exact arithmetic
then mathematically
$\bar{s}_i^{(k)}=s_i^{(k)}=\sqrt{1-(c_i^{(k)})^2}$.

Alternatively, we compute the GSVD of the pair $\{B_k,\bar{B}_k\}$:
\begin{equation}
B_k=P_kC_kW_k^T,\ \bar{B}_k=\bar{P}_k S_kW_k^T,
\end{equation}
where $C_k={\rm diag}(c_1^{(k)},\ldots,c_k^{(k)})$ and
$S_k={\rm diag}(s_1^{(k)},\ldots,s_k^{(k)})$, and $P_k$ and $\bar{P}_k$
are as those defined in \eqref{4.1} and \eqref{4.2}.
The approximate generalized singular values are
$\{c_i^{(k)},s_i^{(k)}\}$ or $c_i^{(k)}/s_i^{(k)}$, the corresponding
left approximate generalized singular vectors for $A$ and $L$ are
$U_{k+1}p_i^{(k)}$ and $\widehat{U}_k\bar{p}_i^{(k)}$, respectively,
and the right approximate generalized singular vectors
are $x_i^{(k)}=Z_kw_i^{(k)}=R^{-1}V_kw_i^{(k)}$.

We remark that, in finite precision, the
$c_i^{(k)}$ and $s_i^{(k)}$ in the GSVD of $\{B_k,\bar{B}_k\}$
are computed with no loss of
accuracy and satisfy $(c_i^{(k)})^2+(s_i^{(k)})^2=1$ unconditionally
within working precision, irrespective of their sizes.
In contrast, it is easily justified from \eqref{3.18} that
if either of $B_k$ and $\bar{B}_k$ is ill conditioned then
$(1-(c_i^{(k)})^2)^{1/2}$ or $(1-(\bar{s}_i^{(k)})^2)^{1/2}$ as
the singular values of $\bar{B}_k$ or $B_k$ may be
inaccurate when $(c_i^{(k)})^2=1-O(\epsilon)$ or
$(\bar{s}_i^{(k)})^2=1-O(\epsilon)$, which corresponds to an ill-conditioned
$\bar{B}_k$ or $B_k$ with small singular values no more than
$O(\epsilon^{1/2})$,
respectively. In other words, if $\bar{B}_k$ is ill conditioned and
has a singular value no more than $O(\epsilon^{1/2})$,
one would compute it with no accuracy from the corresponding singular value
of $B_k$, and vice versa. In this case, the corresponding approximate
generalized singular values may be quite inaccurate and even has no accuracy;
for the accuracy and reliability, we suggest to
compute the singular values $c_i^{(k)}$ and $\bar{s}_i^{(k)}$ of
both $B_k$ and $\bar{B}_k$ separately
and use $\{c_i^{(k)}, \bar{s}_i^{(k)}\}$ as approximations to
some extreme $\{c_i,s_i\}$.

For the computation of $x_{i}^{(k)}$, it is shown in \cite{Zha1996}
that the explicit inversion $R^{-1}$ can be avoided by noticing that
\begin{equation}\label{rightsv}
\begin{pmatrix}
A \\ L
\end{pmatrix}x_{i}^{(k)}=QRR^{-1}V_{k}w_{i}^{(k)}=\widetilde{V}_{k}w_{i}^{(k)}.
\end{equation}
Then, solving the corresponding consistent linear system by an iterative solver,
e.g., the LSQR algorithm, we obtain $x_{i}^{(k)}$.

\subsection{Convergence, accuracy and reorthogonalization}\label{sec4.2}

In the presence of rounding errors, the behavior of the JBD algorithm
may deviate from that in exact arithmetic.
As mentioned previously, we only consider the rounding errors in the JBD process,
and assume that the SVDs of $B_{k}$ and $\bar{B}_{k}$ are computed
accurately, so are $x_{i}^{(k)}$, $y_{i}^{(k)}$ and $z_{i}^{(k)}$.

First, we investigate the convergence of the computed generalized
singular values using the SVD of $B_{k}$.
Since $c_{i}^{2}+s_{i}^{2}=1$, in order to compute the generalized singular
value $\{c_{i}, s_{i}\}$, we only need to compute $c_{i}$, which is
a singular value of $Q_{A}$. Note that $c_{i}^{(k)}$,
the singular value of $B_{k}$, is a computed Ritz value
of $Q_{A}$, due to the property that the $k$-step JBD process
for computing $B_{k}$ is Lanczos bidiagonalization applied to $Q_{A}$ with
the rounding error $O(\lVert \underline{B}_{k}^{-1}\lVert\epsilon)$; see
Theorem~\ref{thm3.1}. In exact arithmetic, it is shown in \cite{JiaYang2020}
that the eigenvalues of $B_k^TB_k$ are the Ritz values of $Q_A^TQ_A$ with respect
to the Krylov subspace generated by $Q_A^Tb, (Q_A^TQ_A)Q_A^Tb,\ldots,(Q_A^TQ_A)^{k-1}
Q_A^Tb$ and the $k$-step lower bidiagonalization of $Q_A$ is equivalent
to the symmetric Lanczos process applied to $Q_A^TQ_A$ and the starting vector
$Q_A^Tb/\|Q_A^Tb\|$ that generates the symmetric tridiagonal matrix
$B_k^TB_k$ and an orthonormal basis matrix $V_k$ of the aforementioned Krylov
subspace. Therefore, the convergence theory of the symmetric Lanczos method
applies (cf. \cite{Parlett1980,Saad1980}), and the singular values of
$B_k$ generally favor some largest and smallest ones of $Q_A$. More convergence
results and details have been given in \cite{jia2020a,jia2020,JiaYang2020}.
In finite precision, typical convergence features of the symmetric
Lanczos method carry over to our case, which include:
(a) the computed Ritz values generally favor the extreme singular values of $Q_{A}$,
while interior singular values are hard to
be approximated; (b) if the extreme singular values of $Q_A$ are better
separated one another, the corresponding computed Ritz values converge more rapidly.
As a result, if we compute $c_{i}$ by the SVD of $B_{k}$,
the computed Ritz values approximating the extreme generalized singular
values of $\{A,L\}$ will first converge generally.

In finite precision, some of the singular values of $B_{k}$ could be numerically
multiple as the iteration number $k$ increases,
which may produce ghost approximations to some
generalized singular values of $\{A,L\}$. A direct consequence
is that a simple or genuine multiple
generalized singular value of $\{A,L\}$ could be approximated by
numerically multiple computed Ritz values, which could lead
to a convergence delay of computed Ritz values.
This phenomenon could occur only when the computed
Lanczos vectors lose orthogonality and can be avoided by using some types of
reorthogonalization strategies, such as full reorthogonalization
or the more efficient one-sided reorthogonalization \cite{Simon2000}.

By Theorem \ref{thm3.1}, the JBD process for computing $B_{k}$ is
the lower Lanczos bidiagonalization of $Q_{A}$ with the
rounding error $O(\|\underline{B}_{k}^{-1}\|\epsilon)$ that is comparable
to $O(\epsilon)$ whenever $\|\underline{B}_{k}^{-1}\|$
is modest. If the JBD process is
implemented with one-sided reorthogonalization of $v_{i}$ such
that the orthogonality level of $V_{k}$ achieves
$O(\epsilon)$, exploiting the backward error results on
the Lanczos bidiagonalization with one-sided
reorthogonalization \cite[Theorem 5.2 and Corollary 5.1]{Barlow2013},
we can deduce that the computed $B_{k}$ is the exact one
generated by the Lanczos bidiagonalization of a nearby matrix
$Q_{A}+E_k$ with $\|E_k\|=O(\|\underline{B}_{k}^{-1}\|\epsilon)$.
Therefore, by the perturbation theory of the singular
values \cite[Corollary 8.6.2]{Golub2013}, the extreme singular
values of $Q_{A}$ can be computed with the ultimate accuracy
$O(\|\underline{B}_{k}^{-1}\|\epsilon)$, and ghost Ritz values
can be avoided.

It is known that a necessary orthogonality level of the
computed basis vectors is the semiorthogonality for the symmetric
Lanczos process \cite{Parlett1980,Simon1984a} and
Lanczos bidiagonalization \cite{Larsen1998}. Based on
Theorem \ref{thm3.1} and \cite[Theorem 5]{Larsen1998},
it is straightforward to prove that the semiorthogonality
suffices in finite precision.

\begin{theorem}\label{thm4.1}
Assume that the compact QR factorizations of $U_{k+1}$ and
$V_{k}$ are $U_{k+1}=M_{k+1}R_{k+1}$ and $V_{k}=N_{k}S_{k}$,
where the diagonals of $R_{k+1}$ and $S_k$ are positive, and
let $\delta = O(\|\underline{B}_{k}^{-1}\|\epsilon)$. If
$U_{k+1}$ and $V_{k}$ satisfy the semiorthogonality
\begin{equation}\label{4.3}
\xi(U_{k+1}), \  \xi(V_{k}) \leqslant \sqrt{\delta/(2k+1)},
\footnote{In Theorem 5 of Larsen \cite{Larsen1998}, the right-hand side of \eqref{4.3}
is $\sqrt{\delta/k}$
instead of $\sqrt{\delta/(2k+1)}$, but Larsen does not justify it rigorously.
In fact, this result is a corresponding counterpart of \cite[Theorem 4]{Simon1984a}.
Since the $k$-step Lanczos bidiagonalization of $Q_{A}$ with the starting vector
$b$ is equivalent to the $(2k+1)$-step symmetric Lanczos
process \cite[\S 7.6.1]{Bjorck1996} of
	$\bar{C} = \begin{pmatrix}
	0 & Q_{A} \\
	Q_{A}^{T} & 0
	\end{pmatrix}$ with the starting vector
	$\bar{b} = \begin{pmatrix}
	b \\ 0
	\end{pmatrix}$, which holds not only in exact arithmetic
but also in finite precision,
the denominator in \eqref{4.3} should be $2k+1$.}
\end{equation}
then
\begin{equation}\label{4.4}
M_{k+1}^{T}Q_{A}N_{k}=B_{k}+\widetilde{E}_{k}  ,
\end{equation}
where the elements of $\widetilde{E}_{k}$ are $O(\delta)=O(\|\underline{B}_{k}^{-1}\|\epsilon)$
in size.	
\end{theorem}

Notice that $M_{k+1}^{T}Q_{A}N_{k}$ is precisely the Ritz--Galerkin
projection matrix of $Q_A$ with respect to the left and right subspaces
${\rm span}(U_{k+1})$ and ${\rm span}(V_k)$, whose
singular values are the {\em true}
Ritz values of $Q_A$ with respect to these two subspaces in
exact arithmetic, while the singular values of $B_k$ are the
{\em computed} Ritz values only when semiorthogonality is met.
Theorem \ref{thm4.1} indicates that once
the orthogonality levels of $U_{k+1}$ and $V_{k}$ are below
$(\delta/(2k+1))^{1/2}$, the computed Ritz values are close to those
true ones within $O(\epsilon)$ provided that $\|\underline{B}_k^{-1}\|$
is modest. Since the true Ritz values are never ghosts provided that
no breakdown occurs before iteration $k$,
we avoid the appearance of ghost computed Ritz
values whenever $U_{k+1}$ and $V_k$ have semiorthogonality.
Consequently, as long as true Ritz values are approximations
to some singular values of $Q_A$ with the accuracy
$O(\|\underline{B}_{k}^{-1}\|\epsilon)$, the corresponding computed
Ritz values have the same approximation accuracy too. In
the meantime, it is easily justified that,
when $U_{k+1}$ and $V_k$ have full orthogonality
levels $O(\epsilon)$, Theorem \ref{thm4.1} holds with the norm of
the error matrix in the right-hand side still being
$O(\|\underline{B}_{k}^{-1}\|\epsilon)$. Therefore,
the semiorthogonality of $U_{k+1}$ and $V_k$ suffices for
computing generalized singular values accurately.
We have made a detailed investigation on
the JBD process with semiorthogonalization strategy and proposed
an efficient partial reorthogonalization strategy in \cite{JiaLi2021}.

There is a corresponding counterpart of Theorem \ref{thm4.1}
for $\bar{B}_{k}$, as stated blow.

\begin{theorem}\label{thm4.5}
Let $\hat{\delta} =O(c_{4}(m,n,p,k)\epsilon)$ with $c_{4}(m,n,p,k)$
defined by \eqref{c4},
and the compact QR factorizations of $\widehat{U}_{k}$ and $V_{k}$
be $\widehat{U}_{k}=\widehat{M}_{k}\widehat{R}_{k}$ and
$V_{k}=N_{k}S_{k}$, where
the diagonals of $\widehat{R}_{k}$ and $S_{k}$ are
positive. If $\widehat{U}_{k}$ and $\widehat{V}_{k}$ satisfy the
semiorthogonality
$$ \xi(\widehat{U}_{k}), \  \xi(V_{k}) \leqslant \sqrt{\hat{\delta}/(2k+1)},$$
then
$$
\widehat{M}_{k}^{T}Q_{L}N_{k}=\bar{B}_{k}+\bar{E}_{k},
$$
where the elements of $\bar{E}_{k}$ are
$O(\hat{\delta})=O((\|\underline{B}_k^{-1}\|+\|\bar{B}_k^{-1}\|)\epsilon)$ in size.	
\end{theorem}

Comparing Theorem~\ref{thm4.5} with Theorem~\ref{thm4.1}, we find
that Theorem~\ref{thm4.5} requires that $\|\underline{B}_k^{-1}\|$
and $\|\bar{B}_k^{-1}\|$ be controllable, stronger than the requirement of
Theorem~\ref{thm4.1}.

\subsection{Residual norm and stopping criterion}

Now we concentrate on designing an effective and efficient stopping criterion
for the GSVD computation based on the JBD process.
Still, we only assume rounding errors in the JBD process, and suppose that the
other computations are exact.
The following analysis focuses on computing some extreme GSVD components
of $\{A,L\}$ using the SVD of $B_{k}$.

From the GSVD \eqref{2.2} of $\{A,L\}$,  it is known that
the generalized eigenvalues $\lambda$ of the symmetric generalized
eigenvalue problem $A^{T}Ax = \lambda L^{T}Lx$ are
$$\underbrace{\infty, \dots, \infty}_{q}, \ \  \underbrace{(c_{q+1}/s_{q+1})^{2},
\dots, (c_{q+l}/s_{q+l})^{2}}_{l}, \ \ \underbrace{0, \dots, 0}_{t},$$
and the corresponding generalized eigenvectors $x$ are the right generalized
singular vectors $x$ of $\{A, L\}$. The generalized eigenvalue problem
$A^{T}Ax_i = \lambda_i L^{T}Lx_i$ can be equivalently written as $s_i^2A^{T}Ax_i
= c_i^2 L^{T}Lx_i$.
Based on this equivalence, Zha in \cite{Zha1996}
uses the residual norm
\begin{equation}\label{resnorm}
\|r_{i}^{(k)} \| = \| ((s_{i}^{(k)})^{2}A^{T}A-(c_{i}^{(k)})^{2}L^{T}L)x_{i}^{(k)}\|
\end{equation}
to design a stopping criterion for an approximate generalized singular value
pair $\{c_{i}^{(k)}, s_{i}^{(k)}\}$ and the corresponding right vector
$x_{i}^{(k)}$, where $(c_{i}^{(k)})^{2}+(s_{i}^{(k)})^2=1$.
Clearly, the computation of $\|r_i^{(k)}\|$ is expensive since
it needs to compute $x_i^{(k)}$ explicitly by solving
a large scale least squares problem like (\ref{rightsv}) at each iteration
$k$ until the convergence occurs, where $w_i^{(k)}$ depends on
the approach used in the JBD method; see Section~\ref{sec4.1}.
In exact arithmetic, Zha \cite{Zha1996} has proved a sharp bound:
\begin{equation}\label{4.6}
\|r_{i}^{(k)} \| \leq \| R \|\alpha_{k+1}\beta_{k+1}|e_{k}^{T}w_{i}^{(k)}|
\end{equation}
with $R$ defined in \eqref{1.1},
so that $\|R\|\alpha_{k+1}\beta_{k+1}|e_{k}^{T}w_{i}^{(k)}|$
can be used to design a stopping criterion if $\|R\|$ or its reasonable estimate
is available. From \eqref{1.1}, for $C=(A^T,L^T)^T$ we have
$\| R \| = \|C\|=\sigma_{\max}(C)$, the largest singular value of $C$. Therefore,
$\frac{\|r_{i}^{(k)} \|}{\|R\|}$ can be regarded as a relative residual norm
of the approximate eigenvalue $(c_i^{(k)}/s_i^{(k)})^2$ and eigenvector $x_i^{(k)}$ of the generalized eigenvalue problem $s_i^2A^TA_i=c_i^2L^TLx_i$.
In finite precision, we can
obtain the following upper bound for $\|r_{i}^{(k)}\|$.
\begin{theorem}\label{thm4.2}
Suppose that the inner least squares problem \eqref{2.3} is solved accurately
and that $x_i^{(k)}=R^{-1}V_kw_i^{(k)}$ with $w_i^{(k)}$ being the right singular
vector of $B_k$ corresponding to its singular value $c_i^{(k)}$.
Then in finite precision it holds that
\begin{equation}\label{4.7}
\big\| [(s_{i}^{(k)})^{2}A^{T}A-(c_{i}^{(k)})^{2}L^{T}L]x_{i}^{(k)}\big\| \leq\| R\|
\left(\alpha_{k+1}\beta_{k+1}|e_{k}^{T}w_{i}^{(k)}|+O(\|\underline{B}_{k}^{-1}\|\epsilon)
\right).
\end{equation}
\end{theorem}
\begin{proof}
From \eqref{3.26}, we have	
$$ Q_{A}^{T}Q_{A}V_{k} = V_{k}B_{k}^{T}B_{k}+\alpha_{k+1}\beta_{k+1}v_{k+1}e_{k}^{T}+G_{k+1}B_{k}+
Q_{A}^{T}F_{k} .$$	
From \eqref{4.1}, we have
$$B_{k}^{T}B_{k}=W_{k}\Theta_{k}^{2}W_{k}^{T}.$$
Notice that $(s_{i}^{(k)})^{2}=1-(c_{i}^{(k)})^{2}$ and $x_{i}^{(k)}=R^{-1}V_{k}w_{i}^{(k)}$.
Using the above two relations and \eqref{1.1}, we obtain
\begin{align}
[(s_{i}^{(k)})^{2}A^{T}A-(c_{i}^{(k)})^{2}L^{T}L]x_{i}^{(k)}
&= [A^{T}A-(c_{i}^{(k)})^{2}(A^{T}A+L^{T}L)]R^{-1}V_{k}W_{k}e_{i} \notag\\
&= R^{T}[Q_{A}^{T}Q_{A}V_{k}W_{k}-(c_{i}^{(k)})^{2}V_{k}W_{k}]e_{i} \notag\\
&=  R^{T}[\alpha_{k+1}\beta_{k+1}v_{k+1}e_{k}^{T}w_{i}^{(k)}+(G_{k+1}B_{k}+
Q_{A}^{T}F_{k})w_{i}^{(k)}], \label{res}
\end{align}
where $e_{i}$ is the $i$-th column of $I_k$.
From Theorem \ref{thm3.1}, we have
$$
\|(G_{k+1}B_{k}+Q_{A}^{T}F_{k})w_{i}^{(k)} \|=O(\|\underline{B}_{k}^{-1}\|\epsilon).
$$
Therefore, by taking norms in \eqref{res}, we prove the desired result.
\end{proof}

Using the same proof approach and based on Theorem~\ref{thm3.3},
it is straightforward to justify that
this theorem still holds when we use the SVD of $\bar{B}_k$ to compute
the approximate generalized singular value $\{(1-(\bar{s}_i^{(k)})^2)^{1/2},
\bar{s}_i^{(k)}\}$ and approximate right generalized singular vector
$x_i^{(k)}=R^{-1}V_kw_i^{(k)}$ with $w_i^{(k)}$ being the
right singular vector of $\bar{B}_k$ corresponding to the singular value
$\bar{s}_i^{(k)}$. Exploiting Lemma~\ref{lem3.1} and \eqref{3.1}--\eqref{3.3},
by a more involved analysis, we can prove
that Theorem~\ref{thm3.3} holds too when we use the GSVD of $\{B_k,\bar{B}_k\}$
to compute the approximate generalized
singular value $\{c_i^{(k)},s_i^{(k)}\}$ and right generalized singular
vector. We omit details here.

Let $(A^T,L^T)^T=C$, and recall the QR factorization \eqref{1.1} of
$C$. We have $\|R\|=\|C\|$.
If we perform Lanczos bidiagonalization on $C$,
several steps, then the largest Ritz value is a reasonably good lower bound for
$\sigma_{1}(C)$. However, we should remind that
only a roughly good upper bound for $\|C\|$ suffices for our use here.
Notice
$$
\|R\|^2=\|C\|^2=\|C^TC\|\leq \|A^TA\|+\|L^TL\|=\|A\|_1\|A\|_{\infty}
+\|L\|_1\|L\|_{\infty},
$$
where $\|\cdot\|_1$
and $\|\cdot\|_{\infty}$ denote the 1-norm and the infinity norm,
which is cheap to compute when $A$ and $L$ are explicitly stored
in a certain sparse format. We then take the square root of the right-hand
side as an estimate for $\|R\|$. Alternatively, it is simpler to
use $\|C\|_1\leq \|A\|_1+\|L\|_1$ or
$\|C\|_{\infty}=\max\{\|A\|_{\infty},\|L\|_{\infty}\}$
as a replacement of $\|R\|$. Since
$e_{k}^{T}w_{i}^{(k)}$ is available from the SVD of $B_{k}$, the quantity
$|\| R\|\alpha_{k+1}\beta_{k+1}|e_{k}^{T}w_{i}^{(k)}|$ can be used
as a reliable stopping criterion, provided that the stopping tolerance for
the (relative) residual norm is not required to achieve
the level of $\epsilon$. Computationally,
we benefit very much from this criterion since we avoid the explicit
computation of $x_i^{(k)}$ before the convergence.
We will numerically confirm the reliability of the criterion.

\section{Numerical experiments}\label{sec5}

We report numerical experiments to justify the results
obtained. All the numerical experiments were
performed on an Intel (R) Core (TM) i7-7700 CPU 3.60GHz
with the main memory 8GB using the Matlab R2017a with the
machine precision $\epsilon = 2.22 \times 10^{-16}$ under the
Miscrosoft Windows 10 64-bit system. For each matrix pair $\{A, L\}$,
we use $b=(1,\dots,1)^{T}\in\mathbb{R}^{m}$ as the starting vector of the JBD process,
and each inner least squares problem \eqref{2.3} is
solved accurately by computing the QR factorization \eqref{1.1}
and $QQ^{T}w$ for a given $w$.

\subsection{Examples for the JBD process in finite precision}

We choose four matrix pairs to confirm the numerical behavior of
the JBD process in finite precision.
We construct the first pair $\{A_c,L_s\}$ as follows:
Take $n=800$ and $C_A=\diag(c)$, $S_L=\diag(s)$, where $c=(\dfrac{3n}{2}, \dfrac{3n}{2}-1, \dots,
\dfrac{n}{2}+1)/2n$ and $s = (\sqrt{1-c_{1}^{2}}, \dots, \sqrt{1-c_{n}^{2}})$.
Let $D$ be the symmetric orthogonal matrix generated by the MATLAB built-in function
$\texttt{D=gallery(`orthog',n,2)}$. We then define $A=C_AD$ and $L=S_LD$. By construction,
the $i$-th generalized singular value of $\{A, L\}$ is $\{c_{i},s_{i}\}$
and the corresponding right vector $x_{i}$ is the $i$-th column of $D$,
and the left generalized singular vectors $p_i^A$ and $p_i^L$ are
the $i$-th column $e_i$ of $I_n$,
$i = 1, \dots, n$. The remaining three pairs use sparse matrices
from \cite{Davis2011}, where
\begin{equation}\label{l1}
L=L_1 = \left(
\begin{array}{ccccc}
1 & -1 &  &  &  \\
& 1 & -1 &  &  \\
&  & \ddots & \ddots &  \\
&  &  & 1  & -1\\
\end{array}
\right)\in \mathbb{R}^{(n-1)\times n}
\end{equation}
with $n=712$, which is the scaled discrete approximation of the first order derivative
operator, and $L=L_{n}=\diag(l)$ with $l=(2n, 2n-1, \dots, n+1)/1000$, $n=3969$.
Some properties of the four test matrix pairs are described in Table \ref{tab1},
where $\kappa(A)=\|A\|\|A^{\dagger}\|$ is the condition number of $A$.

\begin{table}[htp]
	\centering
	\caption{Properties of the test matrices.}
	\begin{tabular}{|l|l|l|l|l|l|}
		\hline
		$A$	 &$m\times n$		 &$\kappa(A)$	  &$L$	    &$p\times n$    &$\kappa(L)$
		\\  \hline
		{\sf $A_{c}$}	 &$800\times 800$	&2.99 &{\sf $L_{s}$} &$800\times 800$  &1.46 	\\  \hline
		{\sf well1850}  &$1850\times 712$   &111.31  &{\sf $L_{1}$}  &$711\times 712$   &453.27
		\\ \hline
		{\sf rdb2048}    &$2048\times 2048$  &2026.80   &{\sf dw2048}   &$2048\times 2048$ &5301.50 \\  \hline
		{\sf c-23} &$3969\times 3969$ &22795.9  &{\sf $L_{n}$}   &$3969\times 3969$ & 1.9995  \\  \hline
	\end{tabular}
	\label{tab1}
\end{table}

\begin{figure}[htp]
	\begin{minipage}{0.48\linewidth}
		\centerline{\includegraphics[width=6.0cm,height=4cm]{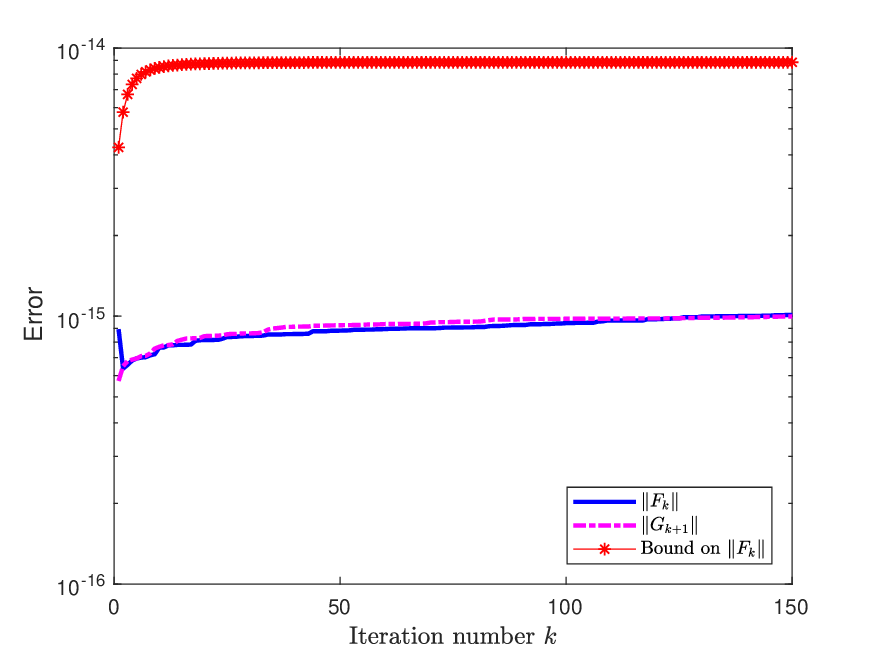}}
		\centerline{(a)}
	\end{minipage}
	\hfill
	\begin{minipage}{0.48\linewidth}
		\centerline{\includegraphics[width=6.0cm,height=4cm]{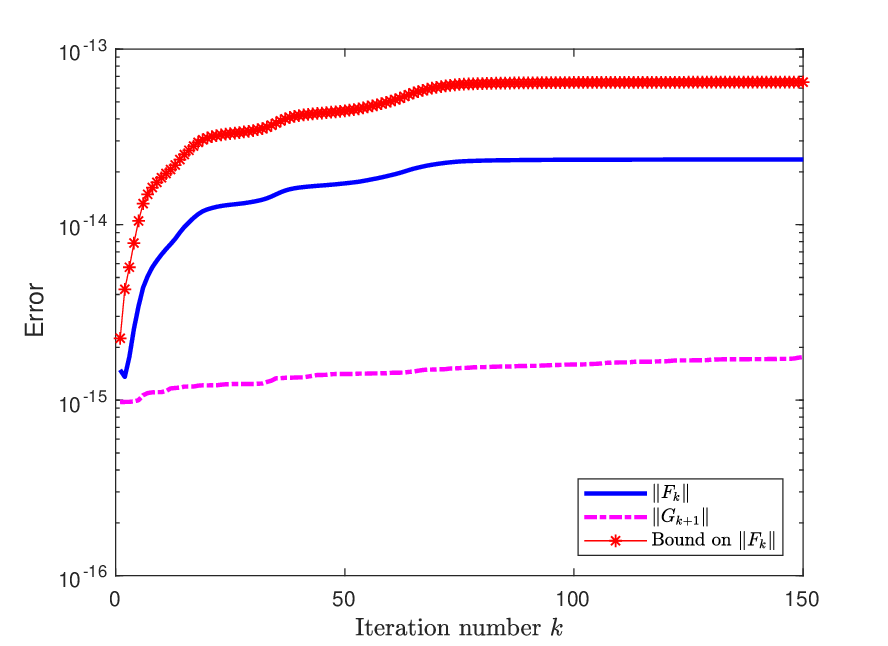}}
		\centerline{(b)}
	\end{minipage}
	
	\vfill
	\begin{minipage}{0.48\linewidth}
		\centerline{\includegraphics[width=6.0cm,height=4cm]{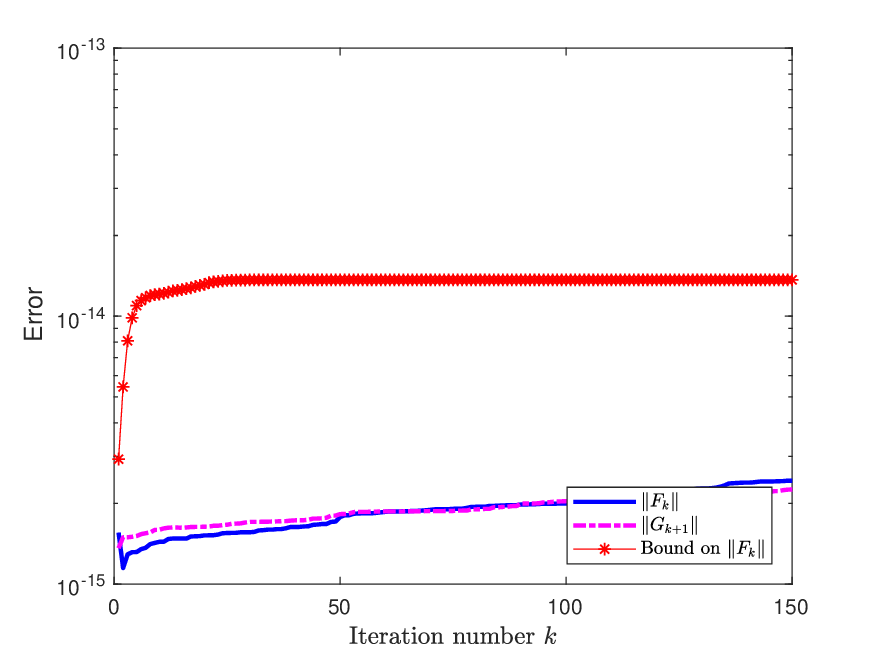}}
		\centerline{(c)}
	\end{minipage}
	\hfill
	\begin{minipage}{0.48\linewidth}
		\centerline{\includegraphics[width=6.0cm,height=4cm]{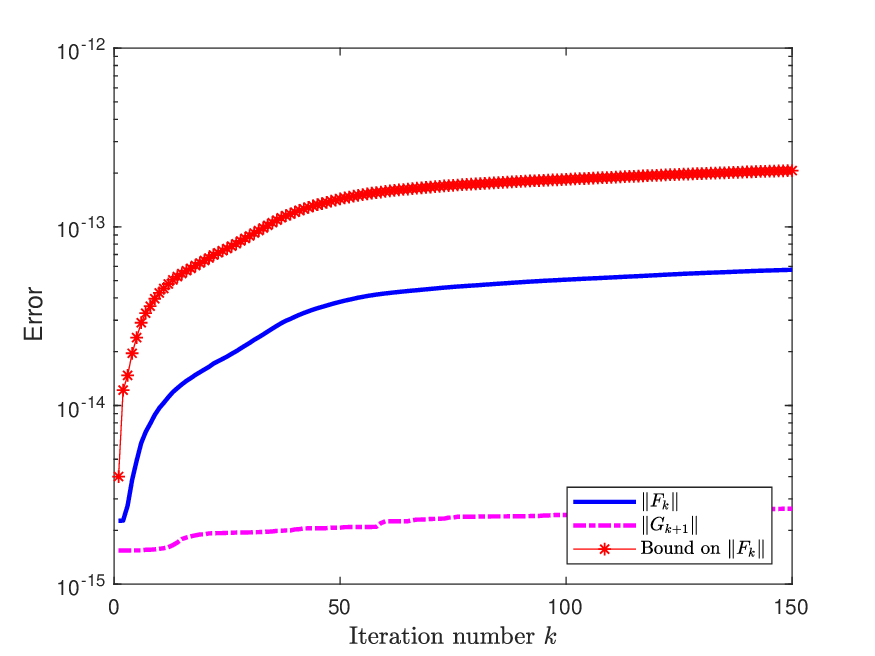}}
		\centerline{(d)}
	\end{minipage}
	\caption{ Estimated error bound for $\|F_{k}\|$: (a) {\sf \{$A_{c}$, $L_{s}$\}};
(b) {\sf \{well1850, $L_{1}$\}}; (c) {\sf \{rdb2048, dw2048\}}; (d) {\sf \{c-23, $L_{n}$\}}.}
	\label{fig1}
\end{figure}

Figure \ref{fig1} depicts the growths of $\|{F}_{k}\|$ and $\|G_{k+1}\|$
in \eqref{3.10} and \eqref{3.11} as the iteration number $k$ increases
from $1$ to $150$.  By Theorem \ref{thm3.1},
we take $O(\lVert \underline{B}_{k}^{-1}\lVert\epsilon)=10\lVert
\underline{B}_{k}^{-1}\lVert\epsilon$. For the four test problems,
it is seen from Figure \ref{fig1}(a)--Figure~\ref{fig1}(d)
that $\|F_{k}\|$ grows
very slowly as $k$ increases. For the four matrix pairs,
$O(\lVert \underline{B}_{k}^{-1}\lVert\epsilon)$ is indeed a very good upper
bound for $\|F_{k}\|$ within ten times, and the growth trends of $\|F_{k}\|$ and
$\lVert \underline{B}_{k}^{-1}\lVert$ are similar.
This indicates that the growth of $\|F_{k}\|$ is mainly affected
by the growth of $\lVert \underline{B}_{k}^{-1}\lVert$.
Since $QQ^{T}\tilde{u}_{i}$ is explicitly computed at each step in our
experiments, $\|G_{k+1}\|=O(\epsilon)$ remains almost unchanged.

\begin{figure}[htp]
	\begin{minipage}{0.48\linewidth}
		\centerline{\includegraphics[width=6.0cm,height=4cm]{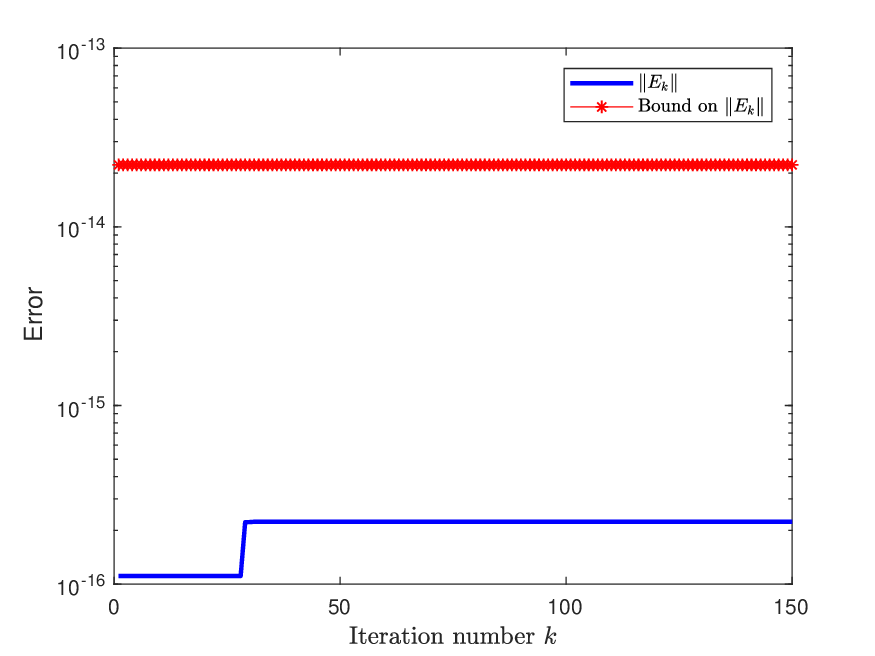}}
		\centerline{(a)}
	\end{minipage}
	\hfill
	\begin{minipage}{0.48\linewidth}
		\centerline{\includegraphics[width=6.0cm,height=4cm]{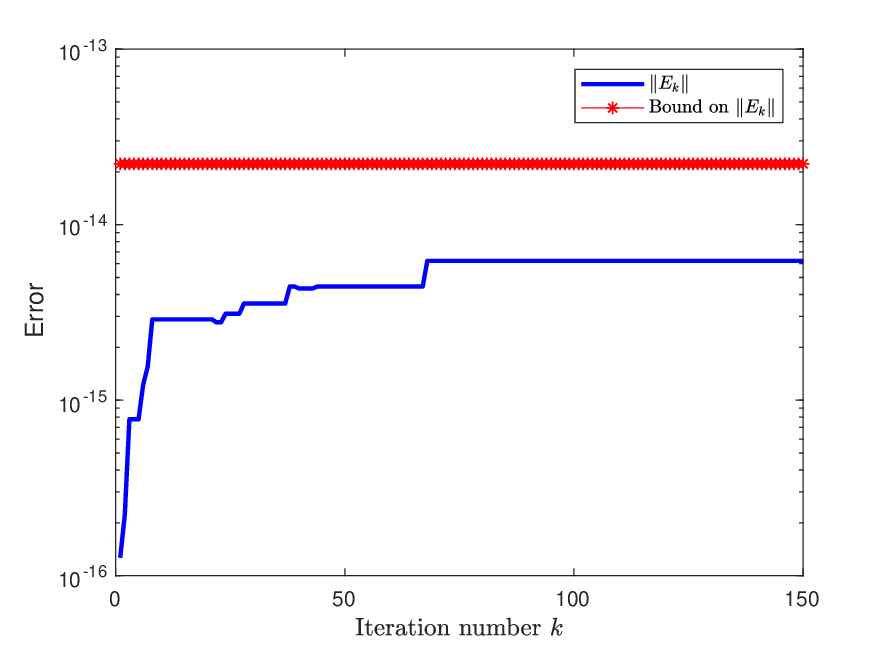}}
		\centerline{(b)}
	\end{minipage}	
	\vfill
	\begin{minipage}{0.48\linewidth}
		\centerline{\includegraphics[width=6.0cm,height=4cm]{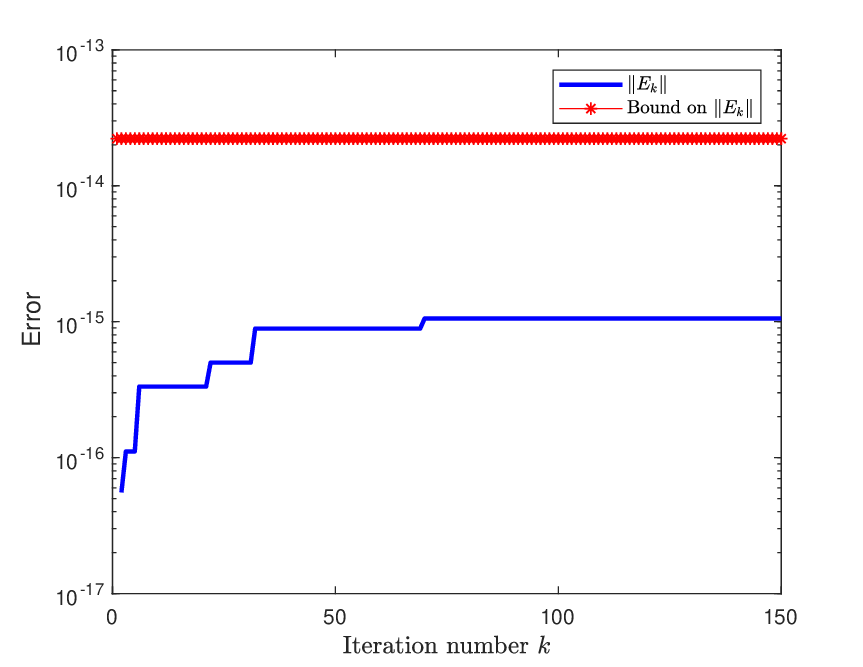}}
		\centerline{(c)}
	\end{minipage}
	\hfill
	\begin{minipage}{0.48\linewidth}
		\centerline{\includegraphics[width=6.0cm,height=4cm]{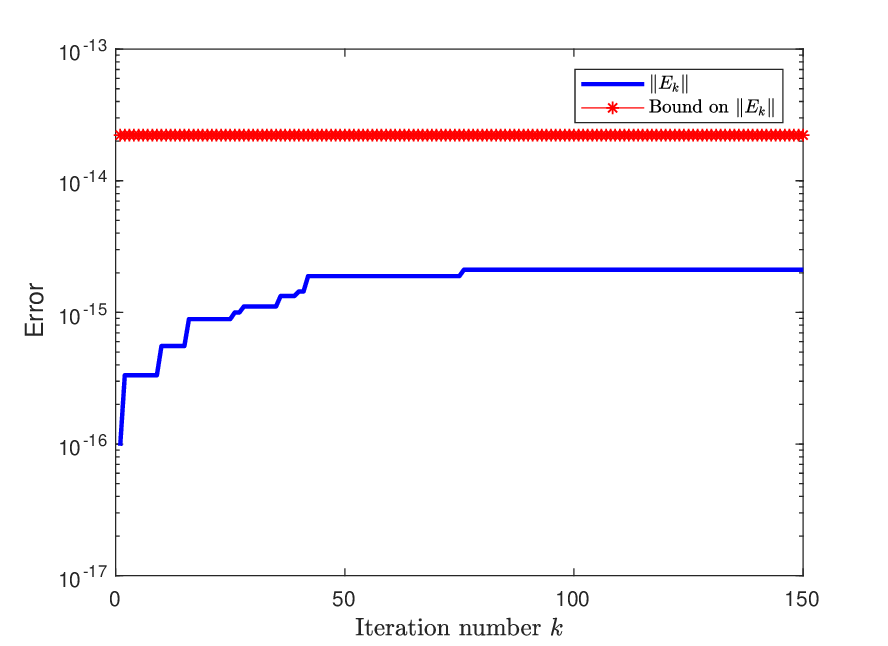}}
		\centerline{(d)}
	\end{minipage}
	\caption{Error bound for $\|I_k-B_k^TB_k-\bar{B}_k\bar{B}_k\|$: (a) {\sf \{$A_{c}$, $L_{s}$\}};
		(b) {\sf \{well1850, $L_{1}$\}}; (c) {\sf \{rdb2048, dw2048\}};
(d) {\sf \{c-23, $L_{n}$\}}.}
	\label{fig2}
\end{figure}

Figure \ref{fig2} depicts the growth of
$\|E_{k}\|=\|I_{k}-B_{k}^{T}B_{k}-\bar{B}_{k}^{T}\bar{B}_{k}\|$.
Since the nonzero elements of $E_{k}$ are $O(c_{3}(m,n,p)\epsilon)$ in size,
we use $100\epsilon$ as an upper bound for $\|E_{k}\|$.
For the four matrix pairs, we find that as $k$
increases from $1$ to $150$, $\|E_{k}\|$ is always at
the level of $\epsilon$ and the bound estimates it quite well,
justifying \eqref{3.18}.

\begin{figure}[htp]
	\begin{minipage}{0.48\linewidth}
		\centerline{\includegraphics[width=6.0cm,height=4cm]{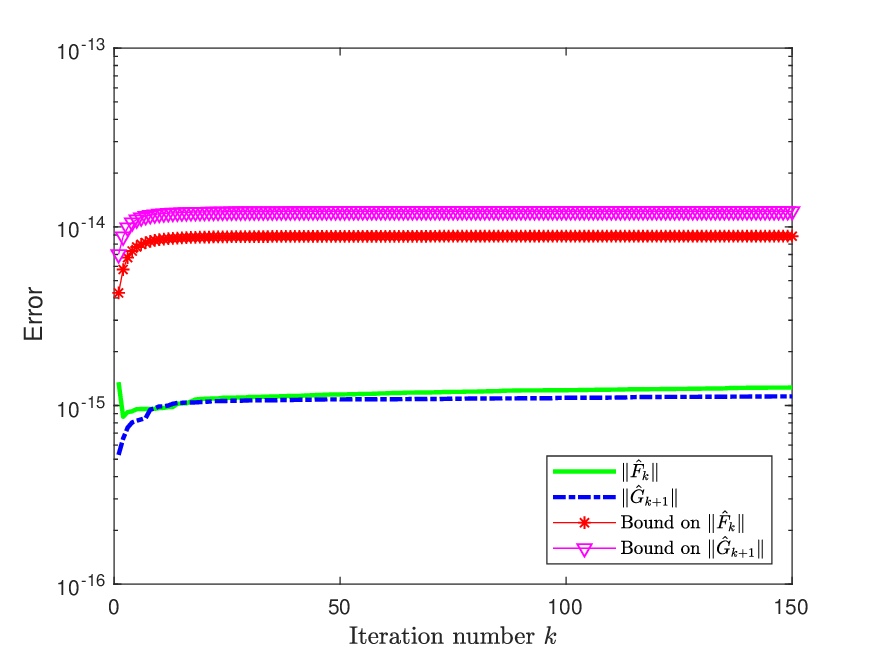}}
		\centerline{(a)}
	\end{minipage}
	\hfill
	\begin{minipage}{0.48\linewidth}
		\centerline{\includegraphics[width=6.0cm,height=4cm]{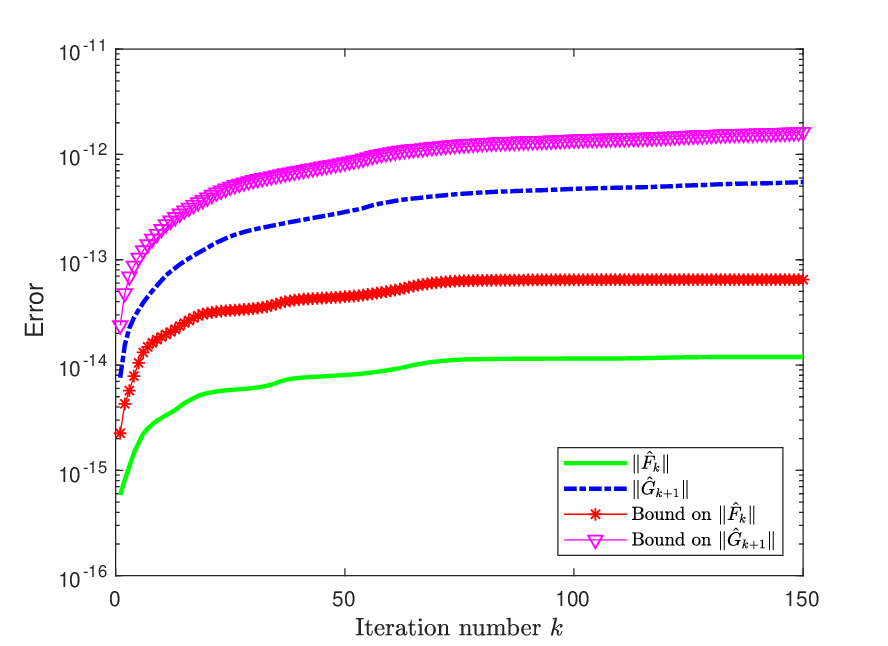}}
		\centerline{(b)}
	\end{minipage}	
	\vfill
	\begin{minipage}{0.48\linewidth}
		\centerline{\includegraphics[width=6.0cm,height=4cm]{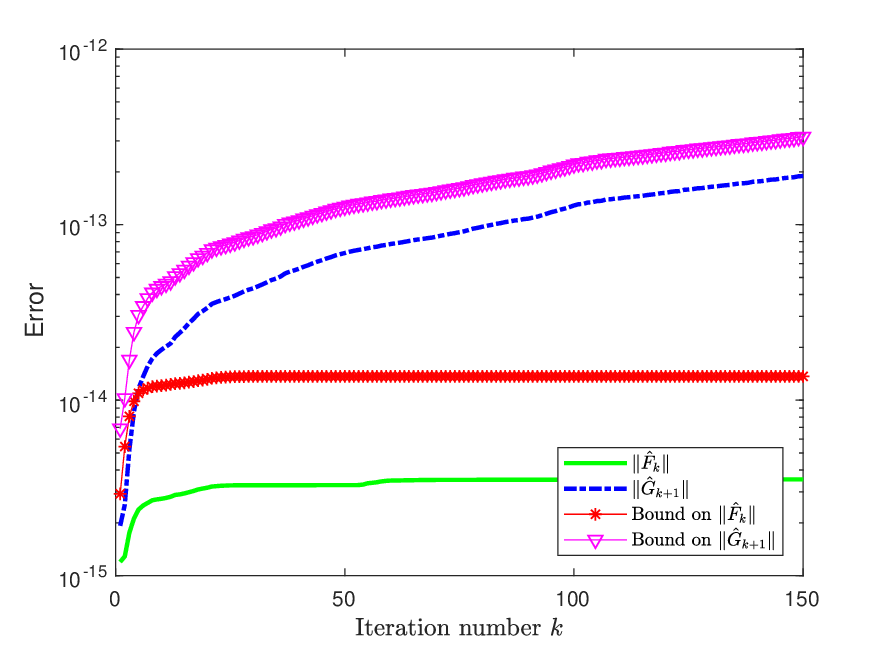}}
		\centerline{(c)}
	\end{minipage}
	\hfill
	\begin{minipage}{0.48\linewidth}
		\centerline{\includegraphics[width=6.0cm,height=4cm]{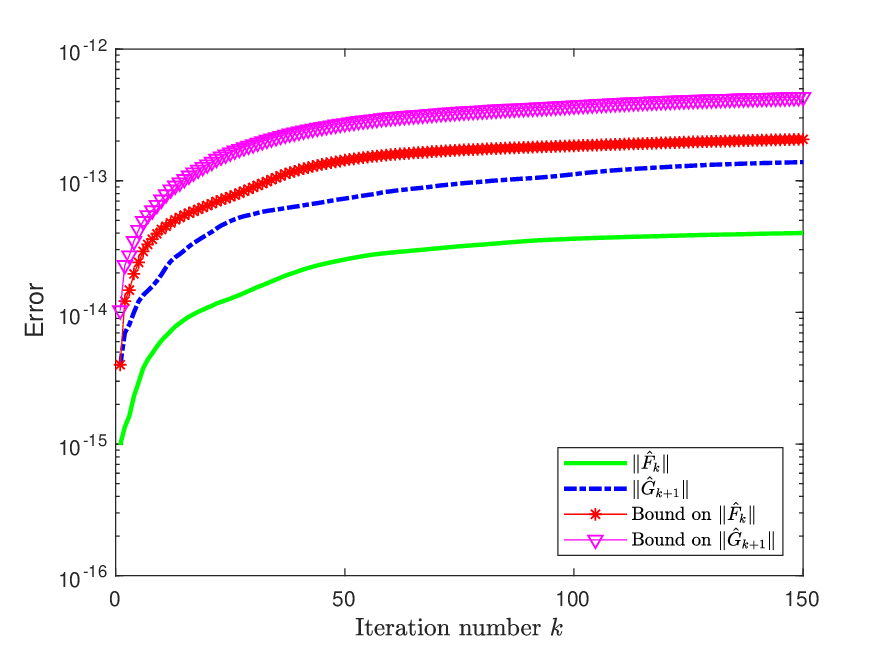}}
		\centerline{(d)}
	\end{minipage}
	\caption{ Estimated error bounds for $\|\widehat{F}_{k}\|$ and $\|\widehat{G}_{k}\|$:
(a) {\sf \{$A_{c}$, $L_{s}$\}}; (b) {\sf \{well1850, $L_{1}$\}};
(c) {\sf \{rdb2048, dw2048\}}; (d) {\sf \{c-23, $L_{n}$\}}.}
	\label{fig3}
\end{figure}

Figure \ref{fig3} depicts the growths of $\|\widehat{F}_{k}\|$ and $\|\widehat{G}_{k}\|$
in \eqref{3.23} and \eqref{3.24}. By Theorem \ref{thm3.3},
we take $O(\lVert \underline{B}_{k}^{-1}\lVert\epsilon)=10\lVert \underline{B}_{k}^{-1}\lVert\epsilon$ and
$O((\lVert \underline{B}_{k}^{-1}\lVert + \|\widehat{B}_{k}^{-1}\|)\epsilon)
=10(\lVert \underline{B}_{k}^{-1}\lVert + \|\widehat{B}_{k}^{-1}\|)\epsilon$,
respectively.
From the figures, we see that $O(\lVert \underline{B}_{k}^{-1}\lVert\epsilon)$
and $O((\lVert \underline{B}_{k}^{-1}\lVert + \|\widehat{B}_{k}^{-1}\|)\epsilon)$
are indeed reasonable upper bounds for $\|\widehat{F}_{k}\|$ and $\|\widehat{G}_{k}\|$,
and the growths of $\|\widehat{F}_{k}\|$
and $\|\widehat{G}_{k}\|$ are critically affected by those
of $\lVert \underline{B}_{k}^{-1}\lVert$ and
$\lVert \underline{B}_{k}^{-1}\lVert + \|\widehat{B}_{k}^{-1}\|$,
respectively. For the four matrix pairs,
$\lVert \underline{B}_{k}^{-1}\lVert$ always grows slowly, but
$\|\widehat{B}_{k}^{-1}\|=\|\bar{B}_k^{-1}\|$ grows faster
for \{well1850, $L_{1}$\} than for the other three pairs. This
is because $L_1$ is truly flat and the smallest
singular value of $\bar{B}_k$ converges to zero as $k$ increases,
causing that $\|\bar{B}_k^{-1}\|$ is ultimately very large,
as we have shown in almost one page after Theorem~\ref{thm3.1}; in
contrast, the smallest
singular value of $\bar{B}_k$ converges to the nonzero smallest
one of $L$ for the other three matrix pairs, and
$\|\bar{B}_k^{-1}\|$ is uniformly bounded by the reciprocal of
the smallest singular value of $L$.

\begin{figure}[htp]
	\begin{minipage}{0.48\linewidth}
		\centerline{\includegraphics[width=6.0cm,height=4cm]{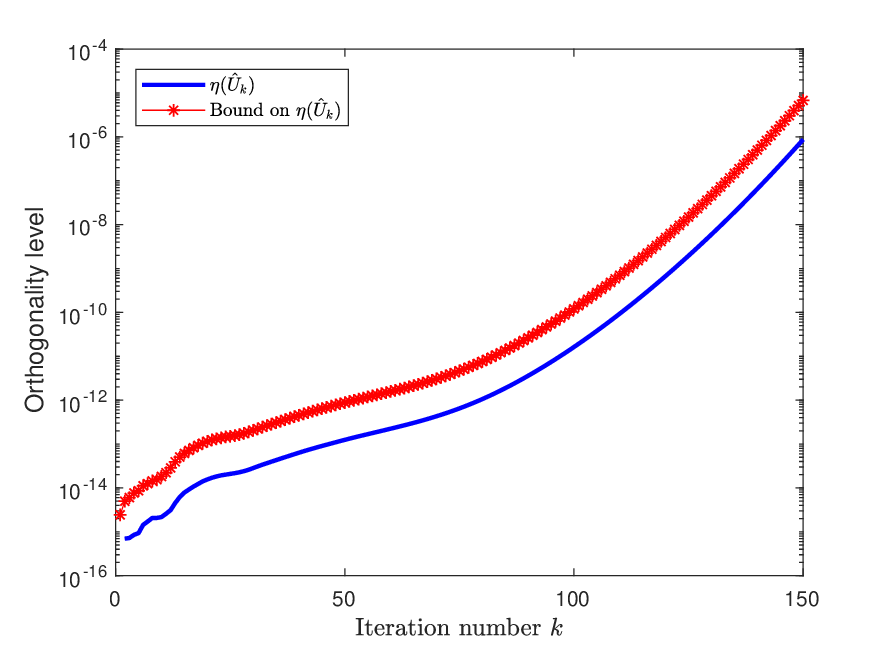}}
		\centerline{(a)}
	\end{minipage}
	\hfill
	\begin{minipage}{0.48\linewidth}
		\centerline{\includegraphics[width=6.0cm,height=4cm]{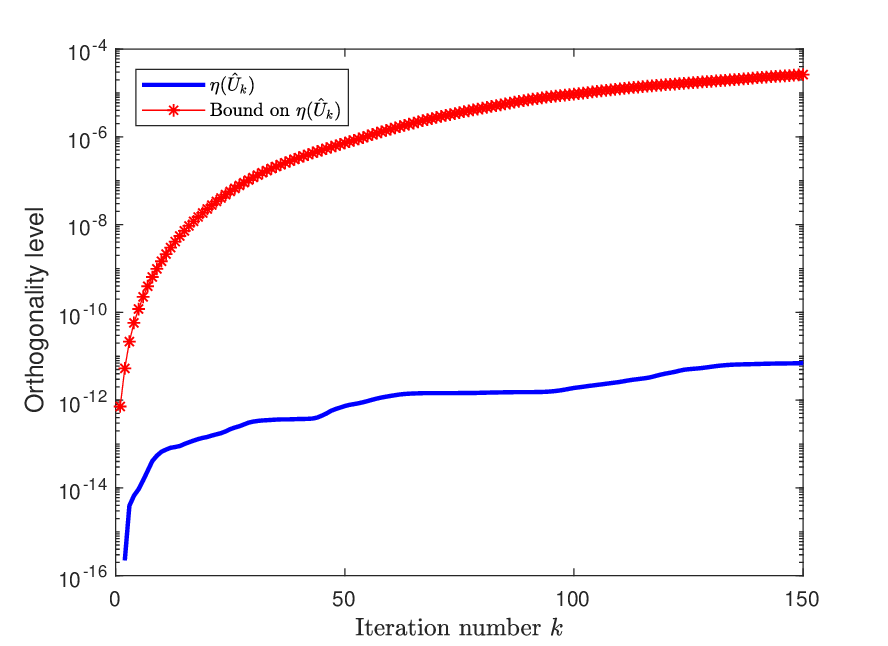}}
		\centerline{(b)}
	\end{minipage}
	
	\vfill
	\begin{minipage}{0.48\linewidth}
		\centerline{\includegraphics[width=6.0cm,height=4cm]{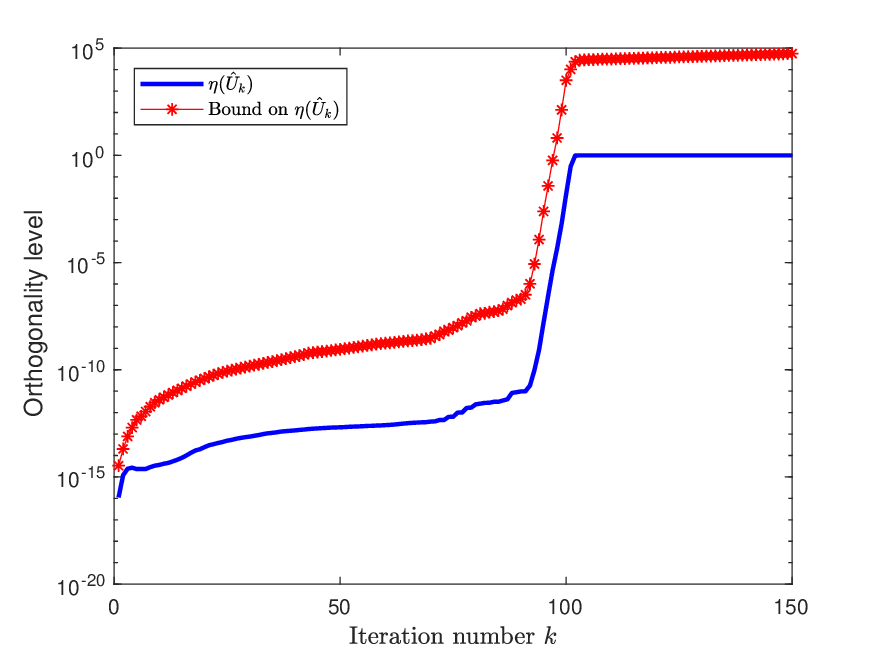}}
		\centerline{(c)}
	\end{minipage}
	\hfill
	\begin{minipage}{0.48\linewidth}
		\centerline{\includegraphics[width=6.0cm,height=4cm]{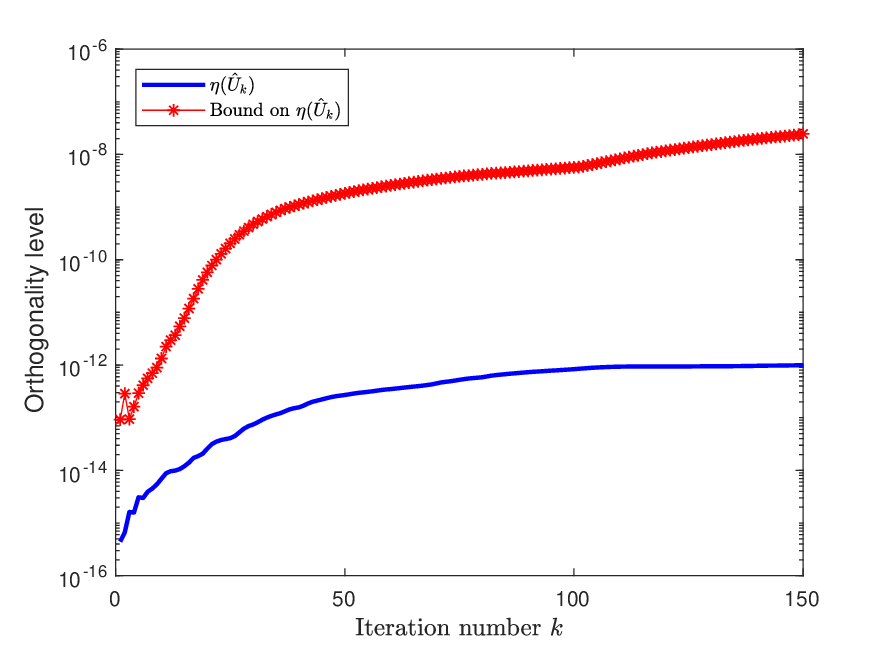}}
		\centerline{(d)}
	\end{minipage}
	\caption{Orthogonality level of $\widehat{U}_{k}$: (a) {\sf \{$A_{c}$, $L_{s}$\}};
(b) {\sf \{well1850, $L_{1}$\}}; (c) {\sf \{rdb2048, dw2048\}}; (d) {\sf \{c-23, $L_{n}$\}}.}
	\label{fig4}
\end{figure}

Figure \ref{fig4} depicts the orthogonality level
of $\widehat{U}_{k}$ measured by $\eta(\widehat{U}_{k})$ as $k$
increases from $1$ to $150$. The upper bound for
$\eta(\widehat{U}_{k})$ is \eqref{3.27}, and
we use $10\epsilon$ as an estimate for $O(c_{3}(m,n,p)\epsilon)$. We observe that
the orthogonality of $\widehat{U}_{k}$ will lose gradually. Particularly,
for the test problem {\sf \{rdb2048, dw2048\}}, the columns of
$\widehat{U}_{k}$ lose orthogonality completely and become numerically linearly
dependent after $k=100$. The growth trends of $\eta(\widehat{U}_{k})$ and
its bound resemble, meaning that the
orthogonality level of $\widehat{U}_{k}$ is affected
not only by $\eta(U_{k})$ and $\eta(\tilde{V}_{k})$
but also by $\|\widehat{B}_{k}^{-1}\|$.

\subsection{Examples for the GSVD computation}

We illustrate the
numerical performance of the JBD method for computing a few extreme
GSVD components of $\{A,L\}$. We will show the
convergence behavior of the computed Ritz values and vectors,
and justify the upper bound in \eqref{4.7}.

\textbf{Example 1.} We show the convergence of the singular values,
the computed Ritz values, of $B_{k}$ or $\bar{B}_{k}$.
Take $m=n=p=500$, and construct a row vector $c=(c_1,\ldots,c_n)$
with
$l_{\max} = 4, l_{\min} = 2,
c_{(1:l_{\max})} = \texttt{linspace}(0.99, 0.7, l_{\max}),
c_{(n-l_{\min}+1:n)} = \texttt{linspace}(0.10, 0.01, l_{\min}),
c_{(l_{\max}+1:n-l_{\min})} = \texttt{linspace}(0.65, 0.15, n-l_{\max}-l_{\min})
$ and
$s = (\sqrt{1-c_1^{2}}, \ldots, \sqrt{1-c_{n}^{2}}),
$
where $\texttt{linspace}$ is the Matlab built-in function. Then define
$C_A=\diag(c)$, $S_L=\diag(s)$ and
$\texttt{D = gallery(`orthog',n,2)}$, and take $A=C_AD$ and $L=S_LD$.
By construction, $\kappa(A)= 6.6000$ and $\kappa(L)=7.0888$,
the $i$-th large generalized singular value pair of $\{A, L\}$ is $\{c_{i},s_{i}\}$.

\begin{figure}[htp]
	\begin{minipage}{0.48\linewidth}
		\centerline{\includegraphics[width=6.0cm,height=4cm]{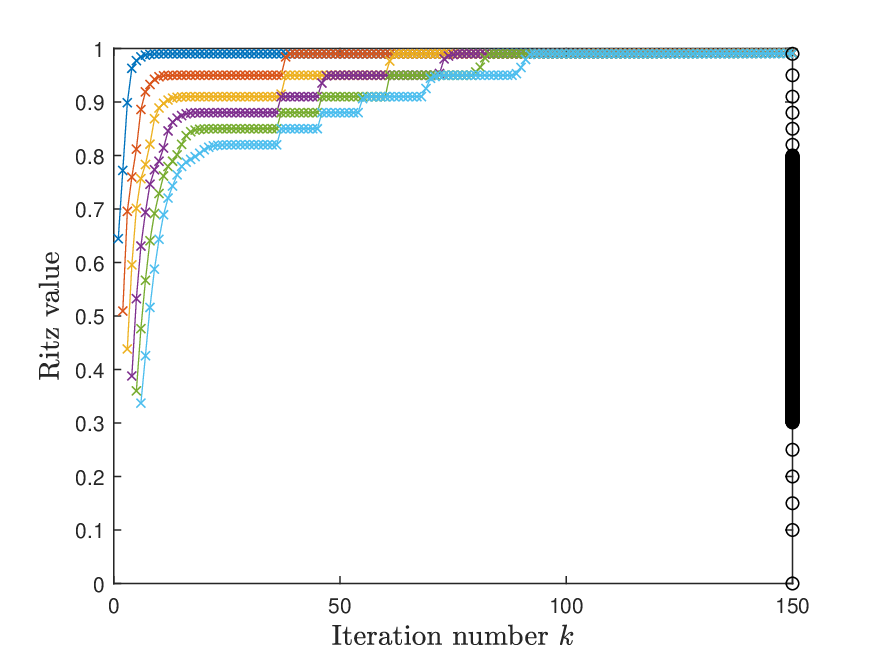}}
		\centerline{(a)}
	\end{minipage}
	\hfill
	\begin{minipage}{0.48\linewidth}
		\centerline{\includegraphics[width=6.0cm,height=4cm]{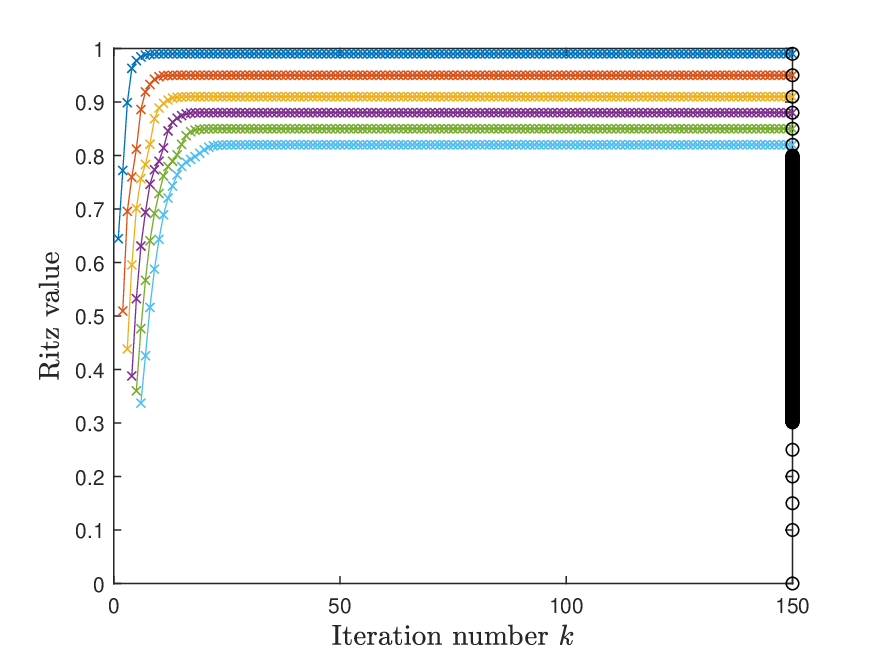}}
		\centerline{(b)}
	\end{minipage}
	
	\vfill
	\begin{minipage}{0.48\linewidth}
		\centerline{\includegraphics[width=6.0cm,height=4cm]{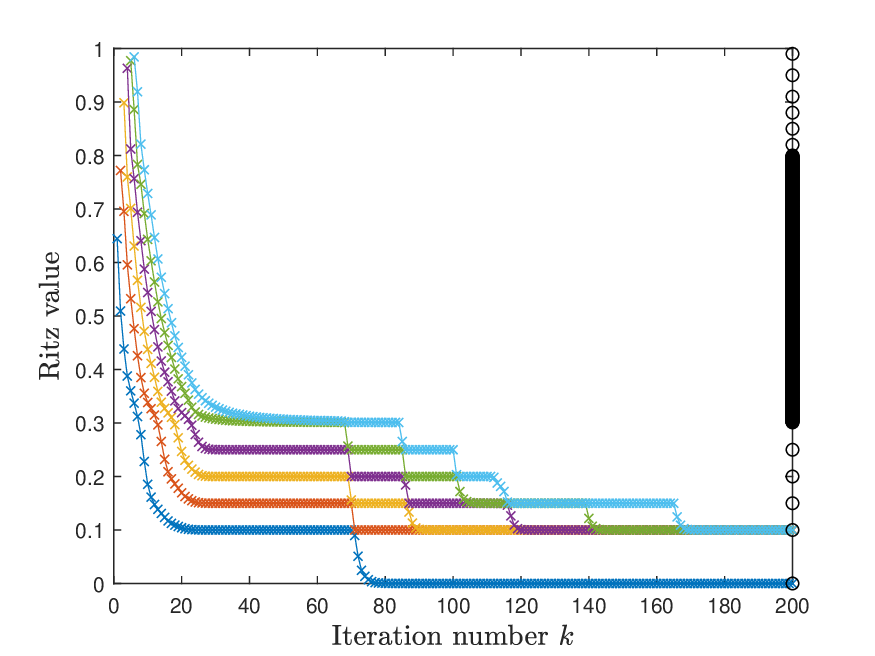}}
		\centerline{(c)}
	\end{minipage}
	\hfill
	\begin{minipage}{0.48\linewidth}
		\centerline{\includegraphics[width=6.0cm,height=4cm]{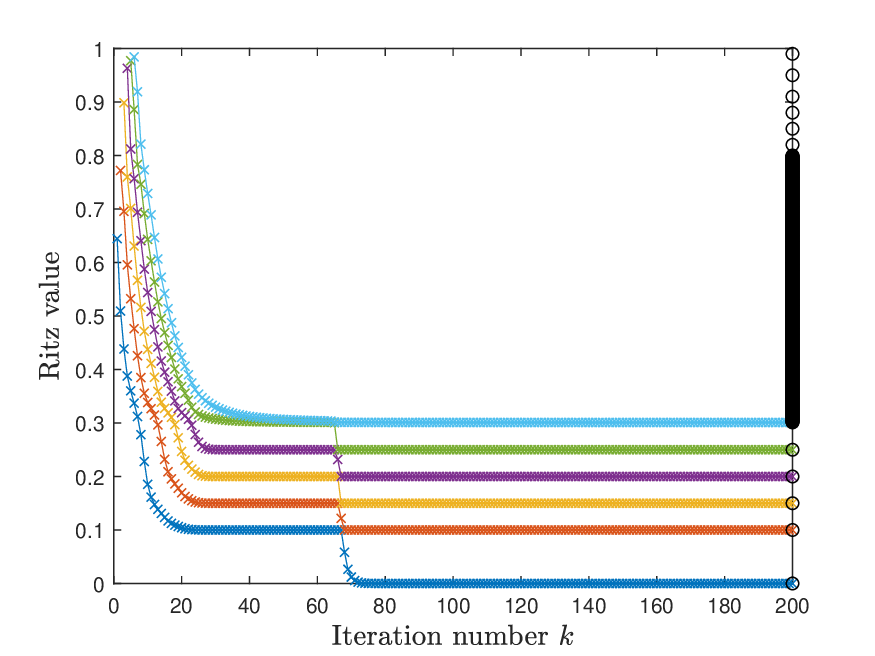}}
		\centerline{(d)}
	\end{minipage}
	\caption{ Convergence of Ritz values computed by the SVD of $B_{k}$:
(a) {the first six largest Ritz values without reorthogonalization};
(b) {the first six largest Ritz values with full reorthogonalization};
(c) {the first six smallest Ritz values without reorthogonalization};
(d) {the first six smallest Ritz values with full reorthogonalization}.}
	\label{fig5}
\end{figure}

Figure \ref{fig5} depicts the convergence processes of the first six largest and
smallest Ritz values computed by the SVD of $B_{k}$, respectively,
where we implemented the JBD process without and with reorthogonalization.
The right vertical line indicates the values of $c_{i}$ for $i=1,\dots, 500$,
and the left and right panels exhibit the convergence behaviors of the JBD method
without and with reorthogonalization.
We observe from Figure~\ref{fig5}(a) and Figure~\ref{fig5}(c)
that some of the converged Ritz values
suddenly jump to become a ghost and then converge to the next large
or small singular values after several iterations. Such a
phenomenon repeats several times and corresponds to spurious copies
each time. More precisely, as Figure~\ref{fig5}(a) indicates,
the six largest Ritz values gradually become numerically
multiple and ultimately
converge to the single largest generalized singular value of $\{A,L\}$
as $k$ increases, of which five ones are ghosts.
Similarly, as Figure~\ref{fig5} (c) shows,
the six smallest Ritz values ultimately converge to the two smallest
generalized singular values of $\{A,L\}$, of which the most slowly converged
Ritz value first converges to the seventh smallest generalized
singular value of $\{A,L\}$, then suddenly jumps and starts to approximate
the smallest one after 69 iterations, and the other five Ritz values
finally become numerically multiple and converge to the second smallest
generalized singular value of $\{A,L\}$ after nearly 170 iterations,
meaning that four of these five ones are ghosts since then.

However, when the JBD process is performed with full reorthogonalization, the
convergence of the Ritz values changes and becomes regular,
as Figure~\ref{fig5}(b) and Figure~\ref{fig5}(d) indicate. In the right panels,
the convergence behavior is much simpler and is in accordance with
theoretical analysis in exact arithmetic. It is clear
that a simple generalized singular value is approximated by Ritz
values without ghosts. We also observe from
both the panels that the extreme Ritz values converge more quickly than
relatively interior Ritz values, that is, the Ritz values closer to
the right-most or left-most generalized singular values stabilize more early,
which confirms the theory that the JBD method generally favors the extreme
generalized singular values.

\begin{figure}[htp]
	\begin{minipage}{0.48\linewidth}
		\centerline{\includegraphics[width=6.0cm,height=4cm]{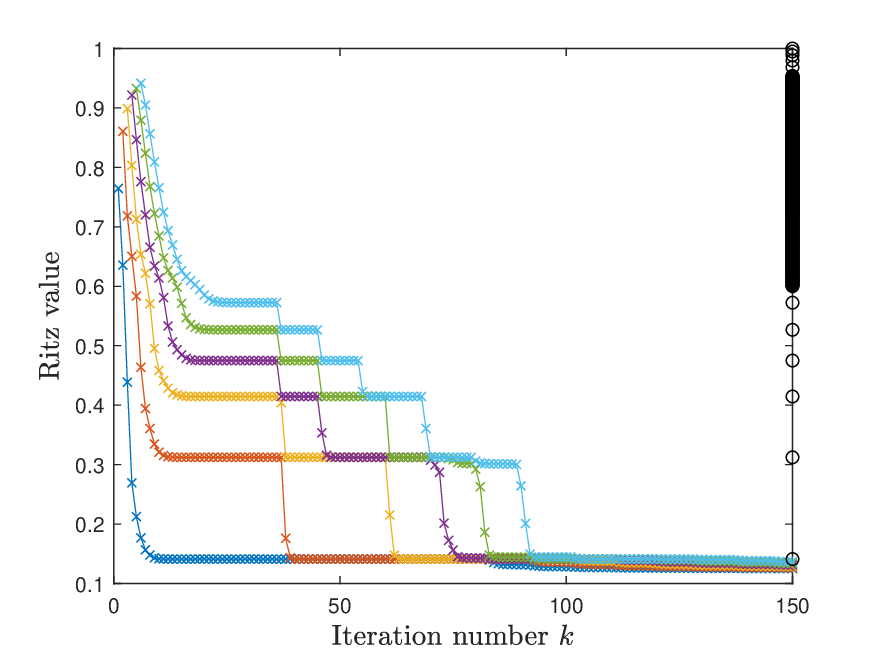}}
		\centerline{(a)}
	\end{minipage}
	\hfill
	\begin{minipage}{0.48\linewidth}
		\centerline{\includegraphics[width=6.0cm,height=4cm]{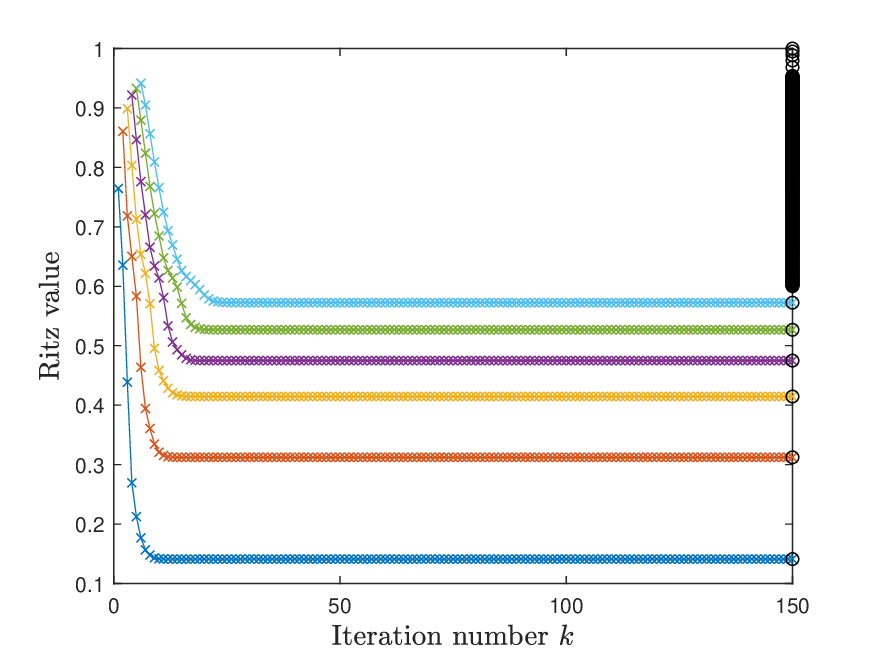}}
		\centerline{(b)}
	\end{minipage}
	\caption{Convergence of Ritz values computed by the SVD of $\bar{B}_{k}$:
(a) {the first six smallest Ritz values without reorthogonalization}; (b)
{the first six smallest Ritz values with full reorthogonalization}.}
	\label{fig6}
\end{figure}

Figure \ref{fig6} depicts the convergence processes
of the first six smallest Ritz values
computed by the SVD of $\bar{B}_{k}$, which corresponds to the first
largest generalized singular values of $\{A,L\}$. The convergence behavior
of the largest Ritz values by the SVD of $\bar{B}_k$
is similar and thus omitted. The right
vertical line indicates the values of $s_{i}$ for $i=1,\dots, 500$.
From Figure \ref{fig6}(a), we observe the ``ghost" phenomenon that some converged
Ritz values suddenly jump and then converge to the next small singular
values after several iterations. More detailed convergence phenomena are
similar to Figure~\ref{fig5}(a) and Figure~\ref{fig5}(c).
Figure \ref{fig6}(b) shows the convergence of Ritz values with full
reorthogonalization, from which it is clear that the JBD method
converges regularly and there are no spurious copies.
Figure \ref{fig6}(a)--Figure \ref{fig6}(b) demonstrate
that the JBD method favors the extreme generalized singular values.

\textbf{Example 2.}  We investigate the convergence of
the approximate generalized singular values and vectors of $\{A,L\}$,
which are computed by using both the SVDs of $B_{k}$ and $\bar{B}_{k}$
and the GSVD of $\{B_{k},\bar{B}_{k}\}$.
We test two matrix pairs. The first
pair is the problem in Example 1, and the second pair $\{dw256A, dw256B\}$
is an electromagnetic problem with $m=n=p=512$ from
the non-Hermitian Eigenvalue Problem Collection in the Matrix
Market\footnote{https://math.nist.gov/MatrixMarket}, where $\kappa(A)=11490.4$
and $\kappa(L)=3.7328$.

We use the JBD method with full reorthogonalization to compute the
largest generalized singular value and vectors. Instead of the
SVD of the individual $B_k$ or
$\bar{B}_k$, we compute the SVDs of $B_k$ and $\bar{B}_k$ simultaneously,
and take $\{c_{1}^{(k)}, \bar{s}_{1}^{(k)}\}$ to approximate
$\{c_{1},s_{1}\}$, where $c_{1}^{(k)}$ is the largest singular
value of $B_{k}$ and $\bar{s}_{1}^{(k)}$ is the smallest singular
value of $\bar{B}_{k}$. The approximations to the right and
left generalized singular vectors $x_1$ and $p_1^A$ are computed from the SVD of
$B_{k}$, and the approximation to the left generalized singular
vectors $p_1^L$ is computed from the SVD of $\bar{B}_{k}$.
Alternatively, we also compute the largest GSVD components
of these two matrix pairs using the GSVD of $\{B_k,\bar{B}_k\}$ and
obtain the approximation $\{c_1^{(k)},s_1^{(k)}\}$ to $\{c_1,s_1\}$
and the approximations $x_1^{(k)},y_1^{(k)},z_1^{(k)}$ to
$x_1,p_1^A,p_1^L$.

We use the angle error
$$ \sin\theta_{k} = |\bar{s}_{1}^{(k)}c_{1}-s_{1}c_{1}^{(k)}|
\mbox{ or } |s_{1}^{(k)}c_{1}-s_{1}c_{1}^{(k)}|$$
to measure the error between $\{c_{1}^{(k)}, \bar{s}_{1}^{(k)}\}$
or $\{c_{1}^{(k)}, s_{1}^{(k)}\}$
and $\{c_{1},s_{1}\}$ \cite{Sun1983}, where $\theta_k$ denotes
the angle between the vectors
$(c_1,s_1)^T$ and $(c_1^{(k)},\bar{s}_1^{(k)})^T$
or $(c_{1}^{(k)},s_{1}^{(k)})^T$.  For the corresponding
generalized singular vectors, we measure the errors
$$\sin\angle(x_{1}, x_{1}^{(k)}), \ \  \sin\angle(p_{1}^{A}, y_{1}^{(k)}), \ \
\sin\angle(p_{1}^{L}, z_{1}^{(k)}).
$$

\begin{figure}[htp]
	\begin{minipage}{0.48\linewidth}
		\centerline{\includegraphics[width=6.0cm,height=4cm]{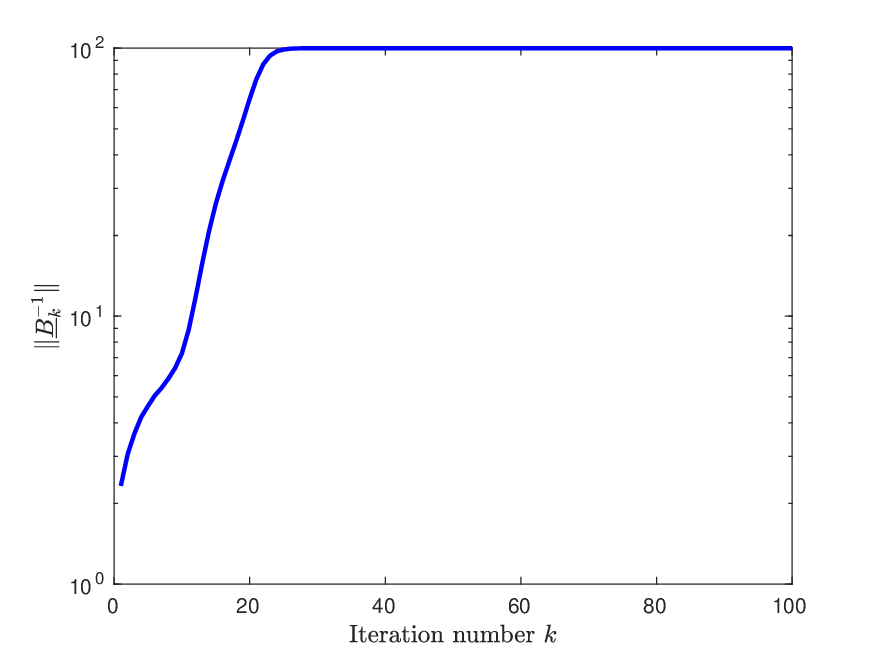}}
		\centerline{(a)}
	\end{minipage}
	\hfill
	\begin{minipage}{0.48\linewidth}
		\centerline{\includegraphics[width=6.0cm,height=4cm]{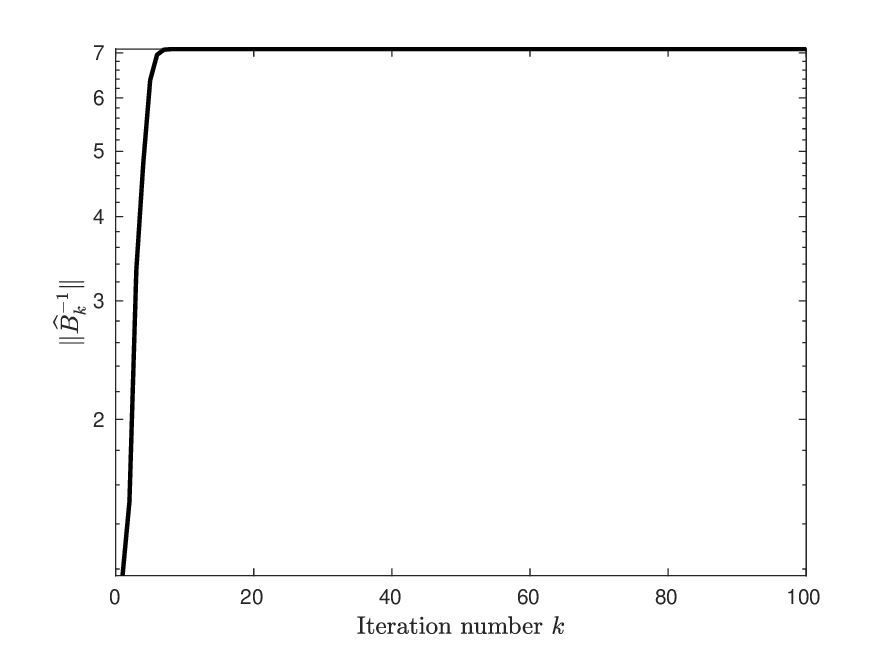}}
		\centerline{(b)}
	\end{minipage}
	
	\vfill
	\begin{minipage}{0.48\linewidth}
		\centerline{\includegraphics[width=6.0cm,height=4cm]{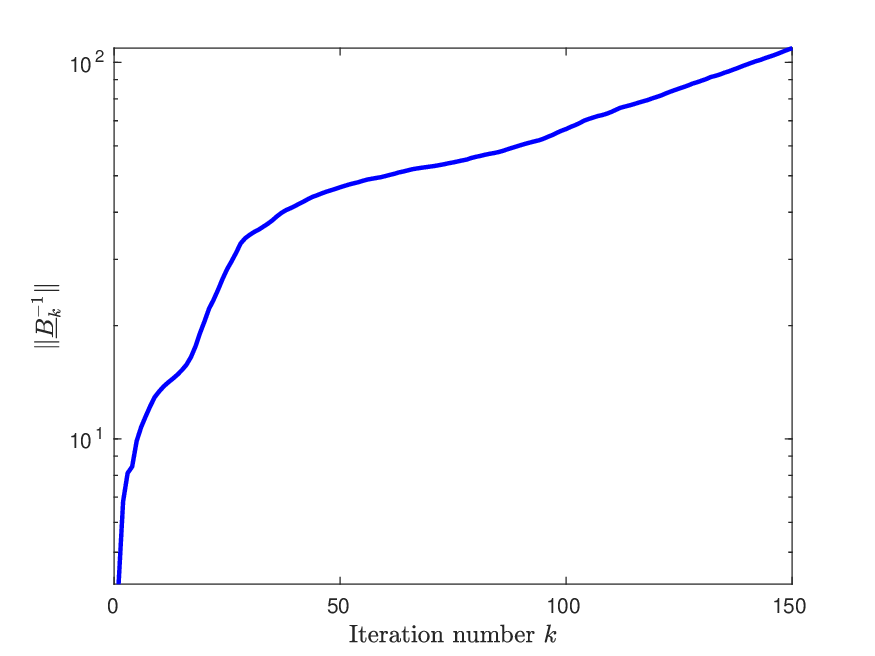}}
		\centerline{(c)}
	\end{minipage}
	\hfill
	\begin{minipage}{0.48\linewidth}
		\centerline{\includegraphics[width=6.0cm,height=4cm]{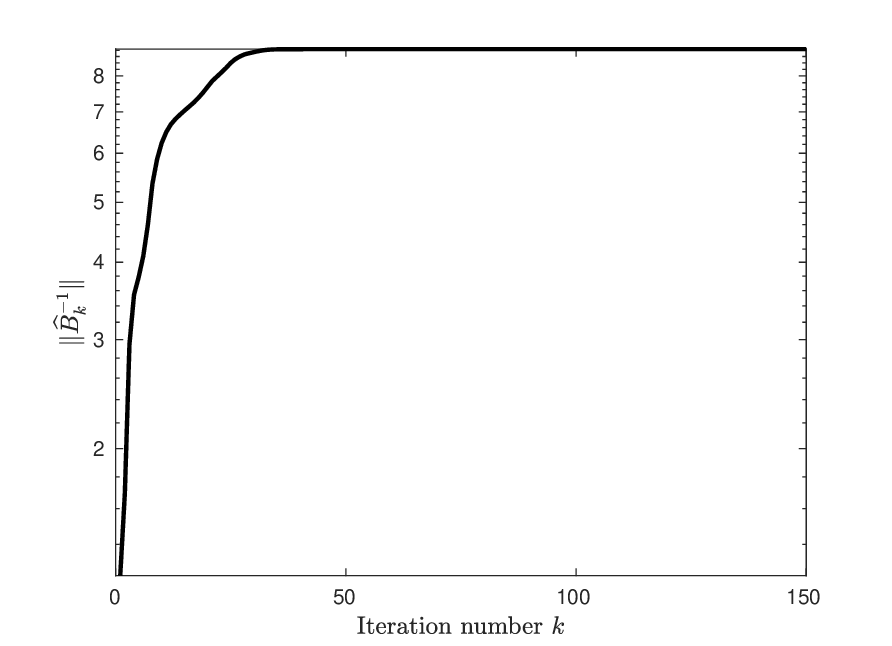}}
		\centerline{(d)}
	\end{minipage}
	\caption{ Growth of $\|\underline{B}_{k}^{-1}\|$ and $\|\widehat{B}_{k}^{-1}\|$:
(a),(b) {\sf \{$A_{500}$, $L_{500}$\}} in Example 1; (c),(d) {\sf \{$dw256A,256B$\}} in Example 2.}
	\label{fig7}
\end{figure}

\begin{figure}[h]
	\begin{minipage}{0.48\linewidth}
		\centerline{\includegraphics[width=6.0cm,height=4cm]{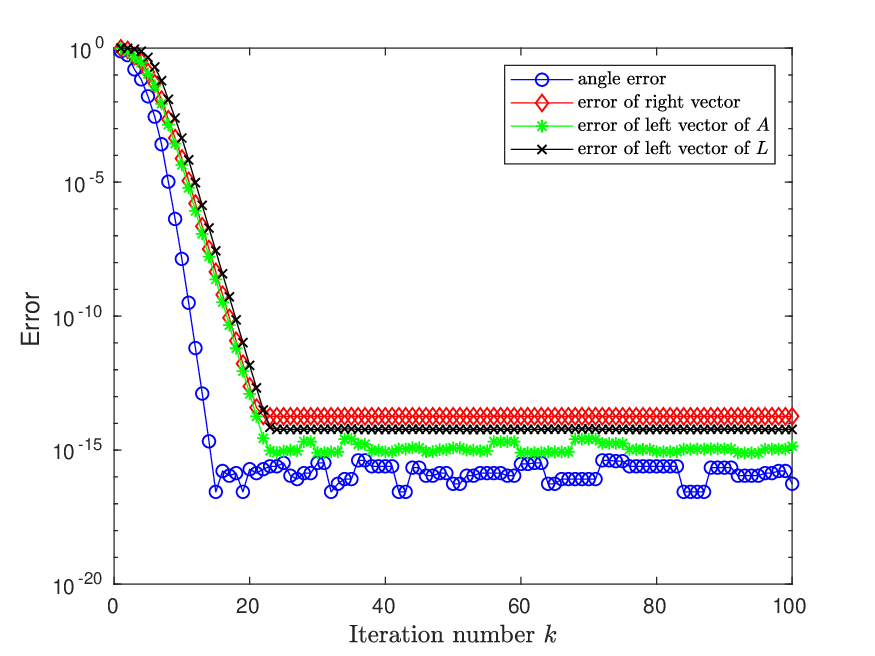}}
		\centerline{(a)}
	\end{minipage}
	\hfill
	\begin{minipage}{0.48\linewidth}
		\centerline{\includegraphics[width=6.0cm,height=4cm]{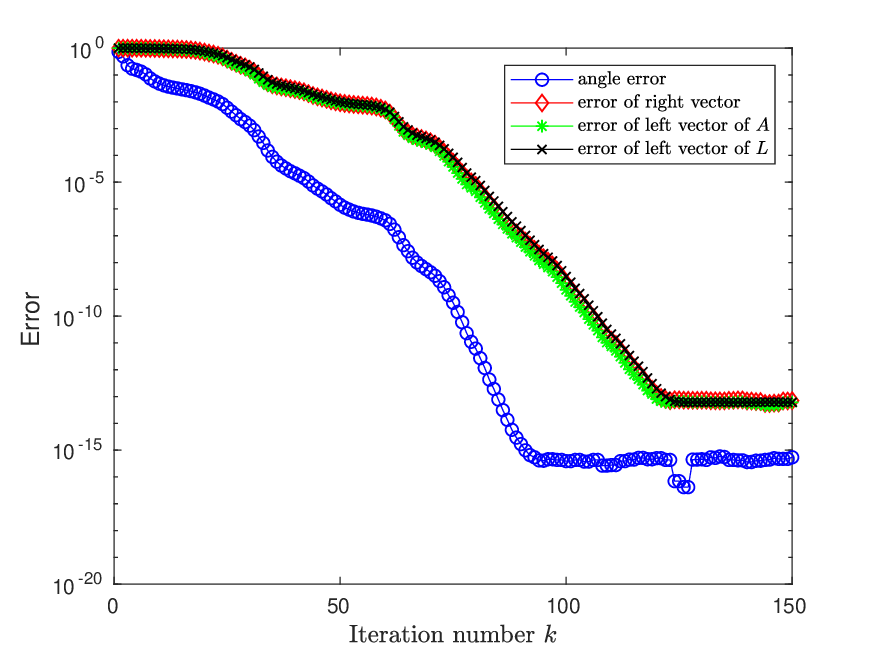}}
		\centerline{(b)}
	\end{minipage}
	\caption{Convergence processes of the approximate GSVD components: (a)
{\sf \{$A_{500}$, $L_{500}$\}} in Example 1; (b) {\sf \{$dw256A,256B$\}} in Example 2.}
	\label{fig8}
\end{figure}

Figure \ref{fig8} draws the approximation processes of the approximate generalized
singular values and vectors obtained by the SVDs of $B_k$ and $\bar{B}_k$
as $k$ increases, while Figure \ref{fig7}
depicts the growths of $\|\underline{B}_{k}^{-1}\|$ and $\|\widehat{B}_{k}^{-1}\|$.
We have found that, for these two matrix pairs, $\|\underline{B}_{k}^{-1}\|$
and $\|\widehat{B}_{k}^{-1}\|$ grow quite slowly and are very modest for Example 1 when
$k=1\sim 150$ and for Example 2 when $k=1\sim 150$, respectively.
The approximate GSVD components converge regularly,
the JBD method converges fast, and all the errors achieve
the level of $\epsilon$ after twenty iterations for Example 1.

Figure~\ref{fig10} depicts the convergence processes
of approximate GSVD components computed by the GSVD of $\{B_k,\bar{B}_k\}$. In this
figure, we also draw the curves of residual norms.
Clearly, the computed results are very similar to those obtained by the SVDs
of $B_k$ and $\bar{B}_k$ until the errors reach the level of $\epsilon$.

\begin{figure}[h]
	\begin{minipage}{0.48\linewidth}
		\centerline{\includegraphics[width=6.0cm,height=4cm]{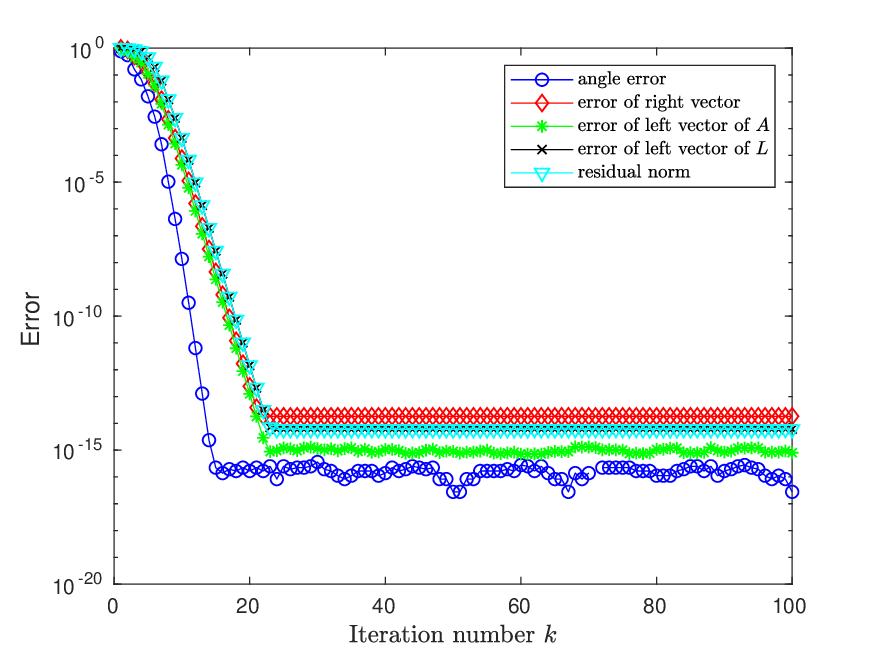}}
		\centerline{(a)}
	\end{minipage}
	\hfill
	\begin{minipage}{0.48\linewidth}
		\centerline{\includegraphics[width=6.0cm,height=4cm]{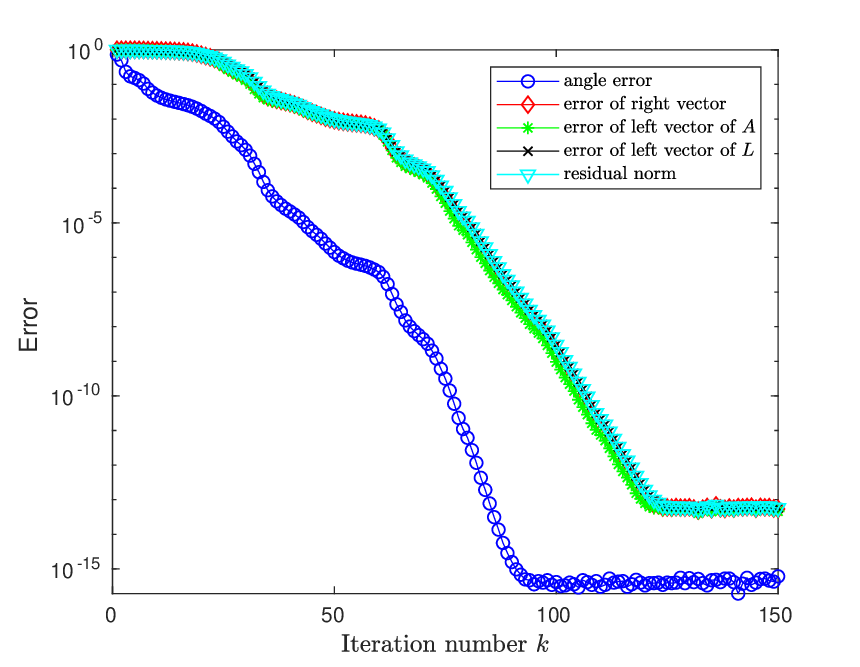}}
		\centerline{(b)}
	\end{minipage}
	\caption{Convergence processes of the approximate GSVD components based on
the GSVD of  $\{B_k,\bar{B}_k\}$: (a)
{\sf \{$A_{500}$, $L_{500}$\}} in Example 1; (b) {\sf \{dw256A,dw256B\}}.}
	\label{fig10}
\end{figure}

\textbf{Example 3.} We show the residual norm and
its upper bound \eqref{4.7}.
The matrix pair $\{A,L\}$ is chosen to be
$\{A_{800}, L_{800}\}$ in Table~\ref{tab1},
and we use the largest singular value of $B_{k}$ to compute
an approximation to the
largest generalized singular value. From the construction,
we have $\|(A_{800}^{T},L_{800}^{T})^{T}\|=1$, and the largest
generalized singular value is $\{c_{1},s_{1}\}$, where $c_{1}=0.75$ and
$s_{1}=\sqrt{1-c_{1}^{2}}$.

We also display the convergence processes of
the approximate generalized singular values by using both the angle error
and relative error
$$ |c_{1}^{(k)}/s_{1}^{(k)}-c_{1}/s_{1}|/(c_{1}/s_{1}).
$$

\begin{figure}[htp]
	\begin{minipage}{0.48\linewidth}
		\centerline{\includegraphics[width=6.0cm,height=4cm]{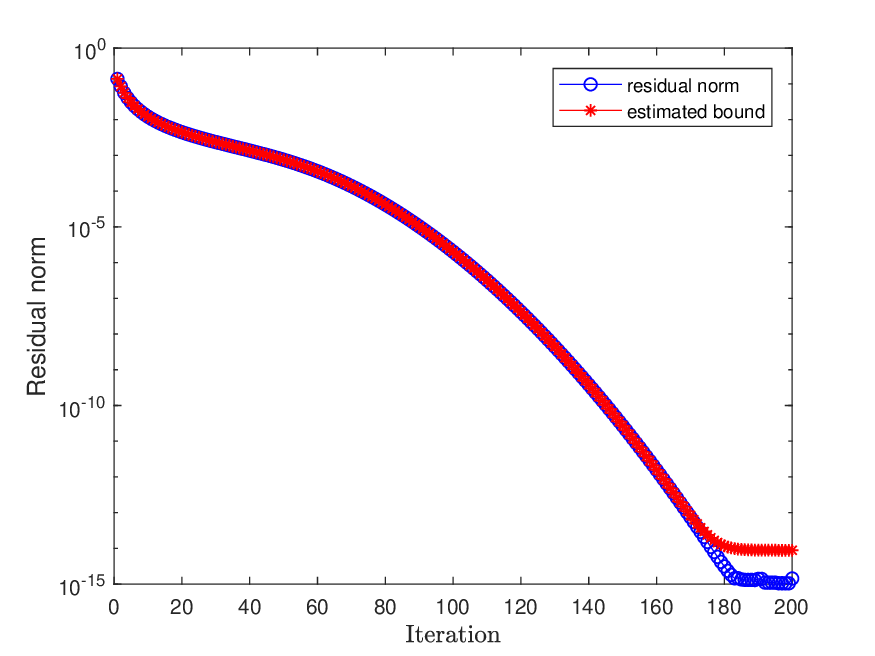}}
		\centerline{(a)}
	\end{minipage}
	\hfill
	\begin{minipage}{0.48\linewidth}
		\centerline{\includegraphics[width=6.0cm,height=4cm]{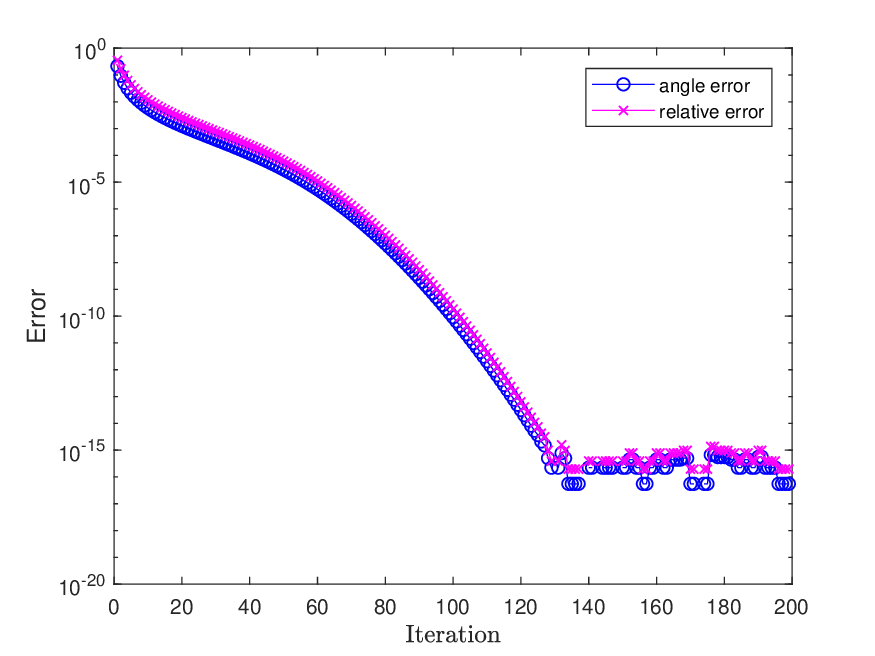}}
		\centerline{(b)}
	\end{minipage}
	\caption{ Convergence history of the approximate largest generalized
singular value of $\{A_{800}, L_{800}\}$: (a) {residual norm and its upper bound};
(b) {angle error and relative error}.}
	\label{fig9}
\end{figure}

We draw the convergence histories of the approximate largest generalized singular value
and the residual norm in Figure \ref{fig9}. From Figure~\ref{fig9}(b),
It is found that the approximate largest generalized singular value
$c_{1}^{(k)}/s_{1}^{(k)}$ converges to
$c_{1}/s_{1}$, and the relative error curve shows that the
approximation accuracy of $c_{i}^{(k)}/s_{i}^{(k)}$ to
$c_{i}/s_{i}$ reaches $O(\epsilon)$. We observe from Figure~\ref{fig9}(a)
that the residual norm and its upper bound are almost the same as $k$
increases. The true residual norm decays until the level of
$\epsilon$, but the estimated upper bound stagnates at the level that
is a little bit higher than $\epsilon$, since the upper bound for
$\|r_{i}^{(k)}\|$ has a term $O(\|\underline{B}_{k}^{-1}\|\epsilon)$,
which is considerably bigger than $\epsilon$ when $\|\underline{B}_{k}^{-1}\|>1$
considerably. For the case
that $\|\underline{B}_{k}^{-1}\|$ remains modest, the
term $O(\|\underline{B}_{k}^{-1}\|\epsilon)$ plays no role in the
upper bound until the bound reaches $O(\epsilon)$. Therefore, the upper bound
$\|R\|\alpha_{k+1}\beta_{k+1}|e_{k}^{T}w_{i}^{(k)}|$ can be used
as a reliable stopping criterion for the JBD algorithm. We have
seen that the angle errors and relative errors resemble very much,
as is expected since $c_1/s_1=O(1)$. We point out
that, in large matrix computations,
a (relative) stopping tolerance is usually $O(\epsilon^{1/2})$. Therefore,
provided
that $\|\underline{B}_{k}^{-1}\|\leq O(\epsilon^{-1/2})$, our upper bound
is a very reliable estimate for $\|r_i^{(k)}\|$.

\section{Conclusions and future work}\label{sec6}

We have made a numerical analysis of the JBD process
on $\{A,L\}$ in finite precision, and
established relationships between it and
respective lower and upper Lanczos bidiagonalizations of
$Q_{A}$ and $Q_{L}$ in the presence of round-offs. The results
have shown that the $k$-step JBD process for computing $U_{k+1}$, $V_{k}$ and $B_{k}$
is equivalent to the lower Lanczos bidiagonalization of $Q_{A}$ with
the error $\delta=O(\lVert \underline{B}_{k}^{-1}\lVert\epsilon)$,
and it for computing $\widehat{U}_{k+1}$, $V_{k}$
and $\widehat{B}_{k}$ is equivalent to the upper Lanczos bidiagonalization
of $Q_{L}$ with the error $\hat{\delta}=O((\|\underline{B}_{k}^{-1}\|+
\|\widehat{B}_{k}^{-1}\|)\epsilon)$.
We have investigated the loss of orthogonality of the computed basis vectors
and established an upper bound for the orthogonality level of $\eta(\widehat{U}_{k})$,
showing that its orthogonality level is controlled by
those of $U_{k+1}$, $\widetilde{V}_{k}$, i.e., $V_k$, and
the quantity $\|\widehat{B}_{k}^{-1}\|$.

We have shown how to use the JBD process to compute a few extreme generalized
singular values and vectors of $\{A,L\}$ and proposed a JBD method
that obtains approximate generalized singular values and vectors
using a few computational approaches, which include
three approaches to obtain
approximate generalized singular values and approximate right generalized singular
vectors by computing the SVDs of $B_k$
and $\bar{B}_k$ and the GSVD of $\{B_k,\bar{B}_k\}$ and two approaches
to obtain approximate left
generalized singular vectors of $A$ and $L$ by computing
either the left singular vectors of $B_k$ and $\bar{B}_k$
simultaneously or the left generalized singular vectors of $\{B_k,\bar{B}_k\}$.
We have considered the convergence and
accuracy of the approximate generalized singular values. The results have
indicated that the generalized singular values of $B_{k}$ and $\bar{B}_k$
are as accurate as the true Ritz values of $Q_A$ and $Q_L$ with respect to
the given subspaces within $\mathcal{O}(\epsilon)$, provided that
the basis vectors have semiorthogonality levels and $\underline{B}_k$
and $\bar{B}_k$ are not ill conditioned. Under these conditions,
it is only necessary to maintain the desired semiorthogonality in order
to obtain the approximate GSVD components with the same accuracy as those obtained by
the JBD method with full reorthogonalization. An
efficient partial reorthogonalization strategy has been proposed
in \cite{JiaLi2021} for this purpose.

In the meantime, we have established a compact
upper bound for the residual norm $\|r_{i}^{(k)}\|$
of an approximate generalized singular value and
approximate right generalized singular vector in finite precision
and shown that it can be used as a cheap and reliable stopping criterion
without explicitly computing the approximate right generalized singular vector
until the convergence occurs.
Finally, we have reported numerical experiments to
justify all the results obtained and assertions.

There remain some important issues. For instance,
due to the limitation of storage,
it is generally necessary to restart the JBD method.
A commonly used restarting technique is the implicit restarting
proposed in \cite{sorensen1992implicit} for the eigenvalue
problem and adapted to the SVD computation in
\cite{jia2003implicitly,jia2010refined,larsen2001combining}.
How to adapt the implicit restart to the JBD method and develop
efficient algorithms is very significant.
Also, notice that the residual norm \eqref{resnorm} is used to measure
the convergence of the JBD method, which is the residual norm of an
approximate generalized eigenpair $((c_i^{(k)}/s_i^{(k)})^2,x_i^{(k)})$
of $s_i^2A^TAx_i=c_i^2 L^TLx_i$ and does not take
approximate left generalized singular vectors $y_i^{(k)}$ for $A$
and $z_i^{(k)}$ for $L$ into account. A much more proper
criterion is to measure the residual norm of the approximate GSVD components
$(c_i^{(k)},s_i^{(k)},x_i^{(k)},y_i^{(k)},z_i^{(k)})$, which,
by the definition \eqref{gsvdv} of GSVD of $\{A,L\}$, is
$$
\|r_i^{(k)}\|=\sqrt{\|Ax_i^{(k)}-c_i^{(k)}y_i^{(k)}\|^2+
\|Lx_i^{(k)}-s_i^{(k)}z_i^{(k)}\|^2+\|s_i^{(k)}A^Ty_i^{(k)}-
c_i^{(k)}L^Tz_i^{(k)}\|^2}.
$$
We need to establish reliable
upper bounds for it in exact arithmetic and in finite precision so as to
design an efficient stopping criterion for the GSVD computation.




\begin{thebibliography}{1}
	
	\bibitem{Barlow2013}
	{\sc J. L. Barlow}, \emph{Reorthogonalization for the Golub-Kahan-Lanczos bidiagonal reduction},
	\newblock Numer. Math., 124 (2013), pp. 237--278.
	
	\bibitem{Bjorck1996}
	{\sc {\AA}.~Bj{\"{o}}rck}, \emph{Numerical {M}ethods for {L}east {S}quares {P}roblems},
	SIAM, Philadelphia, PA, 1996.
	
	\bibitem{Davis2011}
	{\sc T. A. Davis and Y. Hu}, {\em The University of Florida sparse matrix collection},
	ACM Trans. Math. Software, 38 (2011), pp. 1–-25.
	Data available from \url{http://www.cise.ufl.edu/research/sparse/matrices/}.
	
	\bibitem{Golub1965}
	{\sc G. H. Golub and W. Kahan}, \emph{Calculating the singular values
and pseudo-inverse of a matrix}, SIAM J. Numer. Anal., 2 (1965), pp. 205--224.
	
	\bibitem{Golub2013}
	{\sc G. H. Golub and C. F. Van Loan},
	\emph{Matrix Computations}, 4th ed., The Johns Hopkins University Press, 2013.
	
	\bibitem{Hansen1998}
	{\sc P.~C. Hansen}, \emph{Rank-Deficient and Discrete Ill-Posed Problems:
Numerical Aspects of Linear Inversion},
	SIAM, Philadelphia, PA, 1998.
	
	\bibitem{Hansen2010}
	{\sc P.~C. Hansen}, \emph{Discrete Inverse Problems: Insight and Algorithms},
	SIAM, Philadelphia, PA, 2010.
	
	\bibitem{Higham2002}
	{\sc N. J. Higham}, \emph{Accuracy and Stability of Numerical Algorithms}, 2nd ed.,
	SIAM, Philadelphia, PA, 2002.

\bibitem{jia2020a}
{\sc Z.~Jia}, \emph{The low rank approximations and Ritz values in LSQR for linear discrete ill-posed problem},
Inverse Probl., 36 (2020), 045013 (31pp).

\bibitem{jia2020}
{\sc Z.~Jia}, \emph{Regularization properties of LSQR for linear discrete ill-posed problems
in the multiple singular value case and best, near best and general low rank
approximations}, Inverse Probl., 36 (2020), 085009 (38pp).

\bibitem{jia2020b}
{\sc Z.~Jia}, \emph{Regularization properties of Krylov iterative solvers
CGME and LSMR for linear dicrete ill-posed problems with an application
to truncated randomized SVDs}, Numer. Algor., 85 (2020), pp.~1281--1310.
	
	\bibitem{JiaLi2021}
	{\sc Z. Jia and H. Li}, \emph{The joint bidiagonalization process
with partial reorthogonalization}, Numer. Algor., 88 (2021), pp.~965--992.

\bibitem{jia2003implicitly}
{\sc Z.~Jia and D.~Niu}, {\em An implicitly restarted refined bidiagonalization
  {L}anczos method for computing a partial singular value decomposition}, SIAM
  J. Matrix Anal. Appl., 25 (2003), pp.~246--265.

\bibitem{jia2010refined}
{\sc Z.~Jia and D.~Niu}, {\em A refined harmonic {L}anczos bidiagonalization
  method and an implicitly restarted algorithm for computing the smallest
  singular triplets of large matrices}, SIAM J. Sci. Comput., 32 (2010),
  pp.~714--744.

	\bibitem{JiaYang2020}
	{\sc Z.~Jia and Y.~Yang}, \emph{A joint bidiagonalization based algorithm
for large scale general-form Tikhonov regularization},
	Appl. Numer. Math., 157 (2020), pp.~159--177.
	
	
	\bibitem{Kilmer2007}
	{\sc M. E. Kilmer, P. C. Hansen, AND M. I. Espa\~{n}ol}, \emph{A projection-based
approach to general-form Tikhonov regularization},
	SIAM J. Sci. Comput., 29 (2007), pp. 315--330.
	
	\bibitem{Lanczos1950}
	{\sc C. Lanczos}, \emph{An iteration method for the solution of eigenvalue
problem of linear differential and integral operators},
	J. Res. Nat. Bur, 45 (1950), pp. 255--282.
	
	\bibitem{Larsen1998}
	{\sc R. M. Larsen}, \emph{Lanczos bidiagonalization with partial reorthogonalization},
	Department of Computer Science, University of Aarhus, 1998.

\bibitem{larsen2001combining}
{\sc R.~M. Larsen}, {\em Combining implicit restarts and partial
  reorthogonalization in {L}anczos bidiagonalization}, SCCM, Stanford
  University,  (2001).
	
	\bibitem{Meurant2006}
	{\sc G. Meurant  and Z. Strako\v{s}}, \emph{The Lanczos and conjugate gradient
algorithms in finite precision arithmetic},
	Acta Numer., 15 (2006), pp. 471--542.
	
	\bibitem{Paige1971}
	{\sc C. C. Paige}, \emph{The computation of eigenvalues and eigenvectors
of very large sparse matrices},
	PhD thesis, London University, London, England, 1971.
	
	\bibitem{Paige1972}
	{\sc C. C. Paige}, \emph{Computational variants of the Lanczos method for the eigenproblem},
	J. Inst. Math. Appl., 10 (1972), pp.~373--381.
	
\bibitem{Paige1976}	
	{\sc C. C. Paige}, \emph{Error analysis of the Lanczos algorithm for
tridiagonalizing a symmetric matrix},
	J. Inst. Math. Appl., 18 (1976), pp.~341--349.
	
	\bibitem{Paige1980}
	{\sc C. C. Paige}, \emph{Accuracy and effectiveness of the Lanczos
algorithm for the symmetric eigenproblem},
	Linear Algebra Appl., 34 (1980), pp.~235--258.
	
	\bibitem{Paige1981}
	{\sc C. C. Paige and M. A. Saunders}, \emph{Towards a generalized singular
value decomposition},
	SIAM J. Numer. Anal., 18 (1981),  pp.~398--405.
	
	\bibitem{Paige1982}
	{\sc C. C. Paige and M. A. Saunders}, \emph{LSQR: an algorithm for sparse
linear equations and sparse least squares},
	ACM Trans. Math. Soft., 8 (1982), pp.~43--71.
	
	\bibitem{Parlett1980}
	{\sc B. N. Parlett}, \emph{The Symmetric Eigenvalue Problem},
	SIAM, Philadelphia, PA, 1998.
	
	\bibitem{Parlett1979}
	{\sc B. N. Parlett and D. S. Scott}, \emph{The Lanczos algorithm with
selective orthogonalization},
	Math. Comput., 33 (1979), pp.~217--238.
	
	\bibitem{Saad1980}
	{\sc Y. Saad}, \emph{On the rates of convergence of the Lanczos and the block-Lanczos
methods},
	SIAM J. Numer. Anal., 17 (1980), pp.~687--706.
	
	\bibitem{Simon1984a}
	{\sc H. D. Simon}, \emph{Analysis of the symmetric Lanczos algorithm with
reorthogonalization methods},
	Linear Algebra Appl., 61 (1984), pp.~101--131.
	
	\bibitem{Simon1984b}
	{\sc H. D. Simon}, \emph{The Lanczos algorithm with partial reorthogonalization},
	Math. Comput., 42 (1984), pp.~115--142.
	
	\bibitem{Simon2000}
	{\sc H. D. Simon and H. Zha}, \emph{Low-rank matrix approximation using the
Lanczos bidiagonalization process with applications},
	SIAM J. Sci. Comput., 21 (2000), pp.~2257--2274.

\bibitem{sorensen1992implicit}
{\sc D.~C. Sorensen}, {\em Implicit application of polynomial filters in a
  k-step {A}rnoldi method}, SIAM J. Matrix Anal. Appl., 13 (1992),
  pp.~357--385.

\bibitem{Sun1983}
{\sc J.-G. Sun}, \emph{Perturbation analysis for the generalized singular value problem},
	SIAM J. Numer. Anal., 20 (1983), pp.~611--625.
	
	\bibitem{Van1976}
	{\sc C. F. Van Loan}, \emph{Generalizing the singular value decompositions},
	SIAM J. Numer. Anal., 13 (1976), pp.~76--83.
	
	\bibitem{Van1985}
	{\sc C. F. Van Loan}, \emph{Computing the CS and generalized singular value decomposition},
	Numer. Math., 46 (1985), pp.~479--491.
	
	\bibitem{Zha1996}
	{\sc H. Zha}, \emph{Computing the generalized singular values/vectors of
large sparse or structured matrix pairs}, Numer. Math., 72 (1996), pp.~391--417.
		
\end{thebibliography}
\end{document}